\documentclass[aap]{imsart}

\RequirePackage{amsthm,amsmath,amsfonts,amssymb}
\RequirePackage[numbers,sort&compress]{natbib}
\RequirePackage[colorlinks,citecolor=blue,urlcolor=blue]{hyperref}
\RequirePackage{graphicx}

\usepackage{amssymb,mathtools,dsfont,enumerate}
\usepackage{tikz}
\mathtoolsset{showonlyrefs}
\usepackage{mathrsfs}

\startlocaldefs
\theoremstyle{plain}

\newtheorem{theorem}{Theorem}[section]
\newtheorem{lemma}[theorem]{Lemma}
\newtheorem{proposition}[theorem]{Proposition}

\newtheorem{remark}[theorem]{Remark}
\theoremstyle{remark}

\newcommand{\tv}{{\rm TV}}
\newcommand{\sumtwo}[2]{\sum_{\substack{#1 \\ #2}}} % sum with 2 lines 

\newcommand{\ind}{\mathbf{1}}

\newcommand{\E}{\mathds{E}}

\newcommand{\R}{\mathbb{R}}

\newcommand{\cA}{\ensuremath{\mathcal A}} 
\newcommand{\cB}{\ensuremath{\mathcal B}}

\newcommand{\cE}{\ensuremath{\mathcal E}} 
\newcommand{\cF}{\ensuremath{\mathcal F}}

\newcommand{\cN}{\ensuremath{\mathcal N}} 
 
\newcommand{\cP}{\ensuremath{\mathcal P}}

\newcommand{\cT}{\ensuremath{\mathcal T}} 
\newcommand{\cU}{\ensuremath{\mathcal U}}

\newcommand{\bE}{\mathbf{E}}

\newcommand{\bP}{\mathbf{P}}

\newcommand{\bbE}{{\ensuremath{\mathbb E}} }

\newcommand{\bbN}{{\ensuremath{\mathbb N}} } 
 
\newcommand{\bbP}{{\ensuremath{\mathbb P}} } 
 
\newcommand{\bbR}{{\ensuremath{\mathbb R}} } 
 
\newcommand{\bbT}{{\ensuremath{\mathbb T}} }

\newcommand{\bbX}{{\ensuremath{\mathbb X}} }

\newcommand{\gep}{\varepsilon}       % \ge already exists...

\newcommand{\gO}{\Omega}
\newcommand{\gl}{\lambda}

\newcommand{\si}{\sigma} 

\newcommand{\wt}{\widetilde }

    \let\d=\delta  
 \let\g=\gamma       \let\l=\lambda
            
   \let\t=\tau

\let\O=\Omega

\def\({\left(}
\def\){\right)}
 \newcommand{\dd}{\,\text{\rm d}}             % a straight d for differentials

\endlocaldefs

\begin{document}

\begin{frontmatter}

\title{Cutoff phenomenon in nonlinear recombinations}
\runtitle{Cutoff phenomenon in nonlinear recombinations}
\begin{aug}
%%%%%%%%%%%%%%%%%%%%%%%%%%%%%%%%%%%%%%%%%%%%%%%
%% Only one address is permitted per author. %%
%% Only division, organization and e-mail is %%
%% included in the address.                  %%
%% Additional information can be included in %%

%%%%%%%%%%%%%%%%%%%%%%%%%%%%%%%%%%%%%%%%%%%%%%%
\author[A]{\fnms{Pietro }~\snm{Caputo}\ead[label=e1]{pietro.caputo@uniroma3.it}},
\author[B]{\fnms{Cyril}~\snm{Labb\'e}\ead[label=e2]{clabbe@lpsm.paris}
%\orcid{0000-0000-0000-0000}
}
\and
\author[C]{\fnms{Hubert}~\snm{Lacoin}\ead[label=e3]{lacoin@impa.br}}
%%%%%%%%%%%%%%%%%%%%%%%%%%%%%%%%%%%%%%%%%%%%%%
%% Addresses                                %%
%%%%%%%%%%%%%%%%%%%%%%%%%%%%%%%%%%%%%%%%%%%%%%
\address[A]{Department of Mathematics and Physics, Roma Tre University, Largo San Murialdo 1, 00146 Roma, Italy\printead[presep={,\ }]{e1}}

\address[B]{Universit\'e Paris Cit\'e, Laboratoire de Probabilit\'es, Statistique et Mod\'elisation, UMR 8001, F-75205 Paris, France\printead[presep={,\ }]{e2}}

\address[C]{IMPA, Estrada Dona Castorina 110, Rio de Janeiro, Brasil\printead[presep={,\ }]{e3}}
\end{aug}

\begin{abstract}
We investigate a quadratic dynamical system known as nonlinear recombinations. This system models the evolution of a probability measure over the Boolean cube, converging to the stationary state obtained as the product of the initial marginals. Our main result reveals a cutoff phenomenon for the total variation distance in both discrete and continuous time. Additionally, we derive the explicit cutoff profiles in the case of monochromatic initial distributions. These profiles are different in the discrete and continuous time settings. The proof leverages a pathwise representation of the solution in terms of   
a fragmentation process associated to a binary tree. In continuous time, the underlying binary tree is given by a branching random process, thus requiring a more elaborate probabilistic analysis.
\end{abstract}



\end{frontmatter}

\section{Introduction}
\medskip

Consider  $Q$, the transition matrix of an aperiodic, irreducible  Markov chain on a finite state space $\gO$, and let $\pi$ denote the corresponding stationary measure. A classic  result states that the distribution of the Markov chain, starting from an arbitrary initial distribution $\mu$, converges to $\pi$. 
Equivalently, if we define a map $T$ on $\mathcal P(\gO)$ -- the set of probability measures on $\gO$ -- by setting $T(\mu)= \mu Q$, 
then the sequence of iterations of $T$ forms a linear dynamical system converging towards its unique fixed point $\pi$.

\medskip

In our setup a \textit{nonlinear evolution} is the sequence of iterations of a \textit{nonlinear} transformation $T: \mathcal P(\gO)\to \mathcal P(\gO)$. Natural examples of nonlinear evolution include the so-called \textit{nonlinear Markov chains} which are time-inhomogenous Markov chain whose transition rule at time $t$ depends on the distribution of the system at time $t$ \cite{AJDM11}.

\medskip

The study of the cutoff phenomenon in high dimensional  Markov chains has emerged as an intriguing %and well-explored 
area of research in recent decades. This phenomenon captures the abrupt convergence to stationarity, characterized by the total variation distance from equilibrium behaving as an approximate step function
 as the dimension approaches infinity \cite{diaconis1996cutoff,levin2017markov}. 
While a conclusive theoretical framework is yet to be established, the pursuit of a theory regarding cutoff phenomena has paved the way for substantial progress in deepening our comprehension of finite Markov chains' convergence to equilibrium; see, e.g., 
%\cite{lubetzky2010cutoff,lacoin2016mixing,lubetzky2017universality,bordenave2019cutoff,salez2023cutoff}. 
\cite{lacoin2016mixing,lubetzky2017universality,salez2023cutoff}.

\medskip

The analogous phenomenon in the context of nonlinear evolutions such as nonlinear Markov chains has received comparatively little attention. 
This can be attributed to the inherent challenges associated with analyzing fixed point convergence  in nonlinear dynamical systems. To illustrate these difficulties, it is worth noting that even fundamental properties such as the monotonicity of the total variation distance over time, a common feature  in linear systems, are not readily available in nonlinear evolutions. As a consequence, 
it seems that developing a nonlinear counterpart to the analysis of  mixing times will require some significant innovations.

In this paper, we explore the matter %address the problem 
in the simple context of %embark on an exploration of this uncharted territory 
%by focusing %our attention 
%on 
the quadratic system known as nonlinear recombinations. 
The model has its roots in the classical Hardy-Weinberg model of genetic recombination \cite{Gei,Hardy,Weinberg}, and has been more recently revisited within the general context of quadratic reversible dynamical systems \cite{CapSin,Sinetal2,Sinetal}, which provide a combinatorial counterpart to well studied evolutions such as the Boltzmann equation from kinetic theory. The
state of the system at time $t$ is a probability vector over the Boolean cube, %$\Om=\{-1,+1\}^n$, 
interpreted as the distribution of genotypes at time $t$ in a given population, and the evolution is dictated by a simple quadratic recombination
mechanism to be described shortly below. The
nonlinear dynamics converges to the product distribution in which all bits are independent, with marginal probabilities
at  each position given by those in the initial distribution at time zero. Quantitative statements about the convergence to stationarity in total variation distance were previously obtained in \cite{Sinetal2} for the discrete time model and in \cite{CapPar} for the continuous time case. In this paper we provide a more detailed analysis of the convergence, and demonstrate the presence of a cutoff phenomenon in both discrete and continuous time models.

 \subsection{Discrete time dynamics}
Fix $n\in\bbN$, let $\O=\{-1,1\}^n$ denote the Boolean cube, 
 and let $\cP(\O)$ denote the set of probability measures on $\O$. 
We write $[n]=\{1,\dots,n\}$ for the set of coordinates and $A\subset [n]$ for a generic subset thereof. Given $A\subset [n]$ and $\mu\in\cP(\O)$,  $\mu_A$ denotes the marginal of $\mu$ on $A$, 
that is the push forward of $\mu$
for the natural projection $\si\mapsto\si_A$ of $\O$ onto $\{-1,1\}^A$. 
  Given $\mu,\nu\in\cP(\O)$ and $A\subset [n]$, the recombination of $\mu,\nu$ at $A$ is defined as the measure $ \mu_A\otimes \nu_{A^c}\in\cP(\O)$ 
 obtained by taking the product of the marginals $ \mu_A,\nu_{A^c}$, where $A^c=[n]\setminus A$ denotes the complement of $A$.
 We then consider the averaged uniform recombination of $\mu, \nu$, defined as
\begin{equation}\label{onestep}
\mu\circ \nu= 2^{-n}\sum_{A\subset[n]}
\mu_A\otimes \nu_{A^c},
\end{equation}
which yields a commutative product in $\cP(\O)$.
In analogy with kinetic theory, we sometimes refer to it as the collision product. Note that $\mu\circ\nu$ may be interpreted as follows. Let $\si,\si'\in\O$
denote two independent random arrays with distribution $\mu,\nu$ respectively, let $A\subset[n]$ be chosen uniformly at random and call $\eta,\eta'\in\O$ the new pair of arrays obtained by swapping the content of the original ones on the set $A$, that is $\eta:=\si_A\si'_{A^c}$ and $\eta':=\si'_A\si_{A^c}$. Then $ \mu\circ\nu$ is the distribution of  $\eta$.

The discrete time dynamics is defined as follows. For every  initial state $\mu\in\cP(\O)$, the state at time $t\in\bbN$ is given by the recursive relations
\begin{equation}\label{recstep}
%\mu_0=\mu\,,\quad 
\mu_t = \mu_{t-1}\circ \mu_{t-1}\,,\qquad %t\in\bbN,\qquad 
\mu_0=\mu.
\end{equation}
We shall also write $\mu_t=T_t(\mu)$, $t\in\bbN$, so that $T_t$ defines a nonlinear semigroup acting on  $\cP(\O)$: $T_0={\rm Id}$, $T_{t+s}=T_tT_s=T_sT_t$, for integers $t,s\ge 0$.
It is not difficult to see that one has convergence in $\cP(\O)$, namely that 
\begin{equation}\label{conv}
T_t(\mu)\longrightarrow\pi_\mu, \qquad  t\to\infty,
\end{equation}
where $ \pi_{\mu}:=\otimes _{i=1}^n \mu_{\{i\}}$ 
is the product measure on $\O$ with the same marginals as $\mu$.
Indeed, each recombination  
preserves the single site marginals and the repeated recombinations produce a fragmentation process, which eventually outputs the fully fragmented distribution $\pi_\mu$. 
We note that the equilibrium distribution depends on the initial state $\mu$, through its marginal distributions, reflecting the conservation law associated to the recombination mechanism. 
To quantify the distance from equilibrium  we use the total variation norm
 \begin{equation}\label{tvdist}
D(\mu,t) = \| T_t(\mu)-\pi_\mu \|_{\rm  TV},
\end{equation}
where the TV distance between two measures $\mu,\nu\in\cP(\O)$ is defined,  as usual, by
 \begin{equation}\label{def:tv}
\| \mu-\nu \|_{\rm  TV} =\, \max_{E\subset \O}\, |\mu(E)-\nu(E)|.
%\tfrac12\sum_{\si\in\O}|
\end{equation}
It is worth noting that, in contrast with the case of ordinary Markov chains, the distance $D(\mu,t) $, for a fixed $\mu\in\cP(\O)$,  is not necessarily a monotone function of $t$. We refer to Remark \ref{rem:monotone} for  counterexamples and a more detailed discussion of this point. As observed in \cite{Sinetal2}, however, for all $\mu\in\cP(\O)$ one has
 \begin{equation}\label{tvdistbound}
D(\mu,t)\leq \tfrac12n(n-1) 2^{-t}\,.
\end{equation}
The above estimate can be obtained by considering the event, in the fragmentation process alluded to above, that full fragmentation has not yet occurred at time $t$. Indeed, at each step, for each $1\leq i<j\leq n$,  there is a probability $1/2$ that sites $i$ and $j$ get separated by the uniform random choice of $A\subset [n]$, namely that $i\in A$ and $j\in A^c$ or viceversa. Therefore, by a union bound, the probability that there exists an unseparated pair of sites at time $t$ satisfies the estimate \eqref{tvdistbound}. In fact, a similar argument provides an upper bound for more general nonlinear recombinations where the uniform distribution $2^{-n}$ in \eqref{onestep} is replaced by a generic distribution $\nu$ over subsets $A\subset [n]$, see \cite[Theorem 5]{Sinetal2}.  
The estimate \eqref{tvdistbound} predicts that $2\log_2n+O(1)$ steps suffice to obtain a small distance to equilibrium. This estimate is rather tight, in the sense that,  if $\mu = \frac12\ind_{-} + \frac12\ind_{+}$ is the monochromatic distribution (all $+1$ or all $-1$ with equal probabilities) then, as shown in  \cite[Theorem 11]{Sinetal2}, one has $D(\mu,t) \geq 1/4$ at time $\frac12\log_2n - C$ for some constant $C>0$. 
Closing the gap between these two bounds, and determining the precise asymptotic behavior of the distance to equilibrium, remained an open problem. A remark after \cite[Theorem 11]{Sinetal2}, suggests that the correct behavior at leading order should be $\log_2n$. 
Noting that the random time associated to full fragmentation can be shown to be concentrated around $2\log_2n$, this says in particular that convergence should occur precisely at one half of the fragmentation time.
Our findings confirm this prediction in a strong sense, and establish the cutoff phenomenon for the nonlinear recombination dynamics defined above. For simplicity of exposition, we state and prove our result in the setting of balanced initial conditions. The result however holds in greater generality, as we discuss in Section \ref{rem:extensions}.  
We define balanced distributions as 
the measures with symmetric marginals, that is we consider the set
$$\cB_n:=\{\mu\in \cP(\O): \; \forall i \in [n],\;\; \mu(\si_i=\pm1)%=\mu(\si_i=-1)
=1/2\}.$$
and define 
 \begin{equation}\label{defdt}
 D_n(t):= \max_{\mu\in\cB_n}\| T_t(\mu)-\pi \|_{\rm  TV},
 \end{equation}
where $\pi$ is now the uniform distribution over $\O$. It is not hard to check that $D_n(\cdot)$ defined in \eqref{defdt} is a non increasing function for each $n\in\bbN$.

\begin{theorem}\label{th:discrete} 
There exists a constant $c>0$ such that, for every integer $n\geq 1$ and every $t\ge 1$, setting $s:=n2^{-t}$, one has
\begin{gather}\label{ubound}
D_n(t)\leq s\,,%\qquad m\le 1/2\,,
\\
D_n(t)\geq 1-2e^{-cs}\,.
%\qquad m\ge 1.
\label{lbound}
\end{gather}
In particular,  if $t_{n,\lambda}=\lfloor \log_2 n +\lambda\rfloor$, $\lambda\in\bbR$, then
\begin{gather}\label{ubound1}
\lim_{\lambda\to\infty}\limsup_{n\to\infty} D_n(t_{n,\lambda})=0,
%\sup_{n\ge 1} D_n(t_{n,\lambda})=0,
\\
\lim_{\lambda\to-\infty}\liminf_{n\to\infty} D_n(t_{n,\lambda})=1.
%\inf_{n\ge 1} D_n(t_{n,\lambda})=1.
\label{lbound1}
\end{gather}
\end{theorem}
In the above statement and in the remainder of the paper we use the notation $\lfloor u\rfloor$ for the integer part.
In the language of Markov chain mixing, equations \eqref{ubound1} and \eqref{lbound1} state that the discrete time nonlinear recombination exhibits cutoff at time $\log_2n$ with window size $O(1)$.
While Theorem \ref{th:discrete} settles the issue for the cutoff phenomenon, it does not specify the precise profile of the function $D_n(\cdot)$ within the cutoff window, that is when  $|t-  \log_2n| = O(1)$. One reason why this profile is difficult to determine is that
the worst initial condition (the one for which the $\max$ in \eqref{defdt} is attained) varies when $t$ fluctuates within the cutoff window.
On the other hand, our asymptotics \eqref{ubound} and \eqref{lbound} are in a sense ``sharp up to constant'' in the limit when $s\to 0$ and $s\to \infty$ respectively (see the discussion in Section \ref{ideas} and Lemma \ref{lem:ubounds2}).
 We content ourselves with the profile of convergence to equilibrium in the monochromatic case, that is when the initial distribution is either all $+1$ or all $-1$ with probability $1/2$.
In this case we adopt the notation 
 \begin{equation}\label{defbardt}
 \bar D_n(t):= \| T_t(\mu)-\pi \|_{\rm TV},\qquad \mu =  \tfrac12\,\ind_{-} + \tfrac12\,\ind_{+},
 \end{equation}
where $\pi$ is the uniform distribution over $\O$. For all $s>0$, write 
\begin{equation}\label{eq:tvnormal}
\varphi(s):=\|\cN(0,1+s)- \cN(0,1) \|_{\rm TV}, 
 \end{equation}
 for the total variation distance between two centered normal random variables with variance $1+s$ and $1$ respectively.
   \begin{theorem}\label{monocase}
% For any constant $M>0$ we have for every
%    \begin{equation}\label{profile}
%  \lim_{n\to \infty}\sup_{t\in [\log_2 n-M,\log_2 n+M]\cap \bbZ}\left| \bar D_n(t) - \varphi\big(n 2^{-t}\big) \right|=0.
% \end{equation}
If $(t_n)$ is a sequence of integers %indexed by a subset of $\bbN$ 
such that
$\lim_{n\to \infty} n2^{-t_n}=s$, then
\begin{equation}\label{perfil}
  \lim_{n\to \infty}\bar D_n(t_n)= \varphi(s).
\end{equation}
\end{theorem}
Roughly speaking, the above result says that if $t_n=\log_2 n+\gl$ then $\lim_{n\to \infty}\bar D_n(t_n)= \varphi(2^{-\gl})$; however, $t_n$ needs to be an integer for $D_n(t_n)$ to make sense and only a discrete $n$-dependent set of values for $\gl$ can be attained.
%The above result says that for $n$ large if  $t=\lfloor\log_2 n+\gl\rfloor$ then  $\bar D_n(t)= \varphi(2^{-\gl})+o(1)$; see Figure \ref{fig:fig} (however, since $t$ is an integer, only a discrete $n$-dependent set of values for $\gl$ can be attained).

%\bigskip
\begin{figure}[h]
	\centering
	\includegraphics[width=7cm]{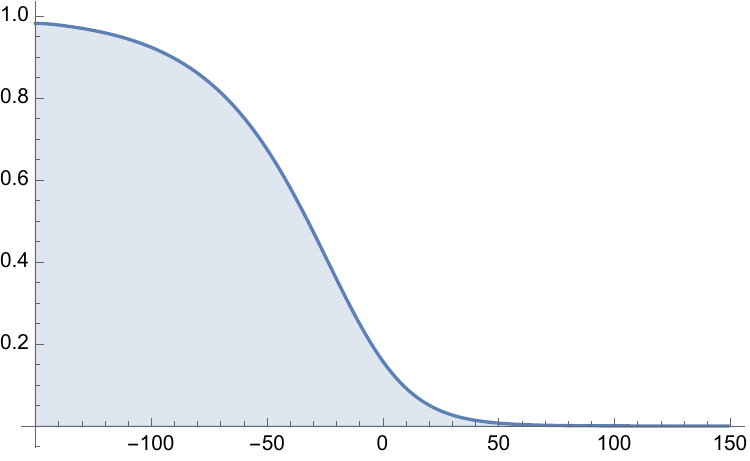}\\
	\caption{Plot of %Sketch of 
	the cutoff profile $\varphi\big(2^{-\lambda}\big)$ %$ \bar D_n(t_{n,\lambda})$ 
	for  $\lambda\in[-150,150]$.}	\label{fig:fig}
\end{figure}

%\bigskip

\subsection{Continuous time dynamics}
We turn to the analysis of the continuous time version of nonlinear recombinations. Given an initial state $\mu\in\cP(\O)$, we define $S_t(\mu)\in\cP(\O)$,  $t\ge 0$, by the differential equation 
\begin{align}\label{cteq}
\frac{d}{dt}\,S_t(\mu) = S_t(\mu)\circ S_t(\mu)\, -\, S_t(\mu)\,, \qquad S_0(\mu)=\mu,
%\sum_{A \subset [n]}\nu(A)\left(p_{t,A}\otimes p_{t,A^c} - p_t\right)
\end{align}
where the symbol `$\circ$' is defined by the averaged uniform recombination in \eqref{onestep}. It can be checked that for each $\mu\in\cP(\O)$ there exists a unique solution to the above Cauchy problem, see e.g.\ \cite{baake2015general} for more general statements in this direction. We will actually provide an explicit construction of such a solution. Thus, equation \eqref{cteq} defines a nonlinear semigroup acting on  $\cP(\O)$: $S_0={\rm Id}$, $S_{t+s}=S_tS_s=S_sS_t$, for all $t,s\ge 0$, which is the natural continuous time analog of the semigroup $T_t$, $t\in\bbN$, defined in \eqref{recstep}. 

We stress, however, that in contrast with linear  Markov chains, the continuous time evolution here cannot be obtained by a simple randomization of the number of steps taken by the discrete time dynamics. This is due to the fact that the collision product defined by \eqref{onestep} does not have the associativity property. For instance, in general one has $$(\mu\circ\mu)\circ(\mu\circ\mu)\neq ((\mu\circ\mu)\circ\mu)\circ\mu.$$
Therefore, the ``collision history'', not only the number of collisions, affects the final distribution. In fact, as we will see, solutions to the continuous time equation \eqref{cteq} admit an explicit representation in terms of a growing random binary rooted tree $\{\cT_t, t>0\}$, which records the history of all collisions contributing to the final state. This is equivalent to the classical representation of the solution of the Boltzmann equation in terms of Wild sums and McKean trees \cite{CCG,mckean1966speed,mckean1967exponential,Wild}. 
The continuous model displays a behavior which is quite different from the discrete one: the value of the mixing time as well as the shape of the profile of convergence to equilibrium are changed.
As in the discrete time model, one has the convergence 
\begin{equation}\label{ctconv}
S_t(\mu)\longrightarrow\pi_\mu, \qquad  t\to\infty,
\end{equation}
to the product measure $\pi_\mu$ with the same marginals as the initial state $\mu$. Quantitative bounds for the speed of convergence have been obtained for the relative entropy  and for the total variation distance \cite{CapPar,CapSin}. In particular, for any $\mu\in \cP(\O)$ one has
%In particular, \cite[Theorem 1.13]{CapPar} shows that for any $\mu\in\cP(\O)$, 
%it is known that after a time of order $\log n$ the distance to stationarity is small.    More precisely, estimates in \cite{CapPar} show that if $t = 
\begin{equation}\label{ctconv2}
\|S_t(\mu)-\pi_\mu\|_{\rm TV}\leq \tfrac12\,n(n-1)\,e^{-t/2},\qquad t>0,
\end{equation}
see Remark \ref{Rk:UpperBd_ct}.
%The rate $1/2$ in the exponent was shown to be optimal.  
Thus, a time $t=4\log n + O(1)$ is sufficient to achieve a small TV distance, regardless of the initial state. Here, we prove that the optimal location is actually $2\log n$ and show that the dynamics displays a cutoff phenomenon. Again, for simplicity, we restrict to balanced initial states $\mu\in\cB_n$, and define  
\begin{equation}\label{defdtct}
 d_n(t):= \max_{\mu\in\cB_n}\| S_t(\mu)-\pi \|_{\rm  TV},
 \end{equation}
where $\pi$ is the uniform distribution over $\O$. 
 
\begin{theorem}\label{th:continuous} 
There exist a constant $c>0$ such that, 
%There exist constants $C,c>0$ such that, 
for all integer $n\geq 1$ and all $t>0$, setting $r:=ne^{-t/2}$, one has 
\begin{gather}\label{uboundct}
d_n(t)\leq r
\\
 d_n(t)\geq 1-2\exp\left( -c \min( \sqrt{r(1\vee \log (1/t))},n) \right)
%\qquad m\ge 1.
\label{lboundct}
\end{gather}
In particular,  if $t_{n,\lambda}:=  2\log n + \lambda$, with $\lambda\in\bbR$, then 
\begin{gather}\label{uboundct1}
\lim_{\lambda\to\infty} \limsup_{n\to\infty}d_n(t_{n,\lambda})=0
%\limsup_{n\to\infty}d_n(t_{n,\lambda})=0,
\\
\lim_{\lambda\to-\infty}\liminf_{n\to\infty}d_n(t_{n,\lambda})=1
%\liminf_{n\to\infty}d_n(t_{n,\lambda})=1.
\label{lboundct2}
\end{gather}
\end{theorem}
We remark that  \eqref{uboundct} is entirely analogous to the upper bound in discrete time \eqref{ubound}. On the other hand, the lower bound \eqref{lboundct}, which is in some sense optimal as discussed in Remark \ref{thevalueoft}, displays a different asymptotics with respect to \eqref{lbound}. Moreover, we note that the logarithmic term in $1/t$ in the exponent allows us to capture some nontrivial behavior for the  time window $t\in (0,1)$ during which $d_n(t)$ drops from $1-2^{-n}$ to $1-e^{-\Theta(\sqrt{n})}$. 
  
As for the analog of Theorem \ref{monocase}, we obtain the following cutoff profile in the case of monochromatic initial states. 
In analogy with \eqref{defbardt}, we write $\bar d_n(t)$ for the distance to stationarity when $\mu = \frac12\ind_{-}+\frac12\ind_+$. %Recall also the definition \eqref{eq:tvnormal} of the function $\varphi$. 

   \begin{theorem}\label{monocasect}
   %
%There exists a positive random variable $W$ with expected value  $\bE[W]=1$ such that, if  
If $t_{n,\lambda}=  2\log n + \lambda$, with $\lambda\in\bbR$, then $\bar d_n(\cdot)$ satisfies 
%%has  the cutoff profile
%\cyril{Je pense que c'est plutôt :
 \begin{equation}\label{profilect}
\lim_{n\to \infty} \bar d_n(t_{n,\lambda}) = f(\lambda)\,,
%\int \frac{e^{-\frac{z^2}{2}}}{\sqrt{2\pi}}\left|\bE\Big[\gamma_{e^{-\lambda/2}W}(z)\Big]-1\right| \dd z
\end{equation}
where $f:\R\mapsto[0,1]$ is a decreasing continuous function with $f(\l)\to 1$, as $\l\to-\infty$, and $f(\l)\to 0$, as $\l\to\infty$.
%\[
%\gamma_u(z):=\frac{e^{\frac{uz^2}{2(u+1)}}}{\sqrt{u+1}}\;,\quad z\in\R\;.
%\]`
% \begin{equation}\label{profilect}
% \lim_{n\to \infty} \bar d_n(t_{n,\lambda}) = \bbE\left[\varphi\big(e^{-\lambda/2}W\big)\right]\,,\qquad \lambda\in\bbR.
%\end{equation}
\end{theorem}
As we will see, the function $f$ appearing in Theorem \ref{monocasect} admits the following more explicit representation:
%Let %$\nu_0=\cN(0,1)$  %denote the standard normal distribution 
 %$\nu_s=\cN(0,1+s)$ denote the centered normal distribution with variance $1+s$. Then, t
 there exists a positive random variable $W$, with expected value $\bE[W]=1$, such that 
 \begin{equation}\label{eq:nulambda}
f(\lambda)=\big\|\cN(0,1)-\bE\big[\cN(0,1+e^{-\lambda/2}W)\big]\big\|_{\rm TV},
\end{equation}
that is $f(\lambda)$ is the total variation distance between the standard normal  and a mixture of normal distributions, defined as the expected value of the centered normal with random variance $1+  e^{-\lambda/2}W$. Note that if  $W$ is replaced deterministically by $1$, then one recovers the profile $\varphi(e^{-\lambda/2})$, a rescaled version of the discrete time result in Theorem \ref{monocase}.  In a sense, 
the random variable $W$ serves as a witness to the presence of additional randomness within the collision history in the continuous time setting, see Proposition \ref{maincase_cont}.
We conclude this introduction with some of the ideas involved in the proofs of the results presented so far, and with some further comments and open questions. 

\subsection{Ideas of proof, comments, and open problems}\label{ideas}
We start with the observation that, in discrete time, by the definitions \eqref{onestep} and \eqref{recstep},  $\mu_t=T_t(\mu)$ can be viewed graphically as 
%the result of averaging over $2^{t-1}$ independent random collisions. Therefore, one may represent $\mu_t$ graphically as 
the average of the distribution at the root of a rooted regular binary tree with depth $t$,
where each internal node represents a random collision and %resulting from the $2^{t-1}$ collisions represented by the internal nodes of the tree, when 
the $2^t$ leaves of the tree are assigned i.i.d.\ samples from $\mu$, which we denote by $\xi=\{\xi(x),\,x=1,\dots,2^t\}$, and the distribution at each internal node $u$ is computed as the collision product of the two distributions at the children of $u$. Indeed, the case $t=1$ was discussed after \eqref{onestep}, and the case of general $t$ is obtained by recursion, see Lemma \ref{lem:rep}. 
%If we denote by $\xi$ the random realization of  we attach 
%distribution of the array $\eta\in\O$ constructed as follows.
%Consider the  rooted binary tree with depth $t$. At each of the $2^t$ leaves we attach a configuration $\xi^u$, $u=1,\dots,2^t$, where the $\xi^u$ are i.i.d.\ with distribution $\mu$, and at each of the $2^{t-1}$ internal nodes of the tree we attach a subset  $A^j\subset [n]$, $j=1,\dots,2^{t-1}$, where the $A^j$ are i.i.d.\ with uniform distribution. For each $i\in[n]$ let $u(i)$ denote the leaf obtained by starting from the root and by descending at each node $j$ to the left or to the right child of that node according to whether $i\in A_j$ or not. 
%Then define $\eta$ as the congiguthe distribution $\mu_t$ is obtained by Given a realization ${\bf \xi}=(\xi^j)_j$ and ${\bf A} = (A^j)_j$ of such variables, 
%the variables $\vec\xi=$ and $A$
%representing the an average over all possible choices of the sets $A$ involved in each recombination. 
Thus we write $\mu_t = \bbE[\mu_t^\xi]$, where $\mu^\xi_t$ is the conditional distribution at the root given the realization $\xi$. %expectation the expectation is with respect to the random variable $\xi$. 
Equivalently, we view $\xi$  as a random environment, and the measure $\mu^\xi_t$ as a quenched distribution. By convexity, an upper bound is obtained by estimating the  TV distance by the average of the quenched %TV distances and by passing to the 
$L^2$ norm:
 \begin{equation}\label{ubidea1}
  \| \mu_t-\pi \|_{\rm  TV}=\tfrac12\,\| \bbE\, (h^\xi_t-1) \|_{L^1(\pi)} \leq \tfrac12\,\bbE\, \| \hat h^\xi_t-1 \|_{L^2(\pi)}\,,
\end{equation}
where $h^\xi_t=\dd\mu^\xi_t/\dd\pi$ is the relative density, $\hat h^\xi_t=h^\xi_t + g^\xi$ for any function $g^\xi$ such that $\bbE[g^\xi]=0$, and $\| \cdot \|_{L^p(\pi)}$ denotes the norm in $L^p(\O,\pi)$. The proof of the  upper bound in Theorem \ref{th:discrete} %starts with this observation and  
then proceeds with an estimation of the quenched $L^2$ norm featuring in \eqref{ubidea1}, which in turn is obtained by an expansion of the shifted density $\hat h^\xi_t$ for a suitable choice of $g^\xi$. 

Before discussing the proof of the lower bound in Theorem \ref{th:discrete}, it is convenient to address first the profile result %, in the special case of the monochromatic distribution; see 
in Theorem \ref{monocase}.    
The proof of this is based on showing that for the monochromatic $\mu$, if $\si\in\O$ is uniformly distributed, then  
$h(\si) = \mu_t(\si)/\pi(\si)$, the density of $\mu_t$ with respect to $\pi$, is well approximated, as $n\to\infty$ and $n2^{-t}\to s$, by the random variable 
 \begin{equation}\label{lbidea1}
\frac1{\sqrt{s+1}}\,\exp{\left(\frac {s\big(\sum_{i=1}^n\si_i\big)^2 }{2n(s+1)}\right)}.
\end{equation} 
%The proof of this requires a detailed expansion for the density $h$. 
Once this is achieved, the representation in terms of normal random variables appearing in \eqref{perfil} follows by using the CLT to replace $(\sum_{i}\si_i)^2/n$ by the square of a standard normal. %$\sim \cN(0,1)$ by the CLT. 
%along with some basic facts about sums of independent Bernoulli random variables.  

The function  $\varphi(s)$ in Theorem \ref{monocase} satisfies $\varphi(s) = \frac{s}{\sqrt{2e\pi}} + o(s)$ as $s\to 0$, see Appendix \ref{app:Gauss}. This shows that the upper bound \eqref{ubound} captures the optimal linear dependence in $s$, as $s$ becomes small. On the other hand, $1-\varphi(s)=\frac{2\sqrt{\log s}}{\sqrt{2\pi s}}(1+o(1))$ as $s\to\infty$, see again Appendix \ref{app:Gauss}. Therefore, the lower bound \eqref{lbound} cannot be achieved by the monochromatic distribution $\mu$ when $s$ is large. Indeed, the proof of the lower bound in Theorem \ref{th:discrete} will use an initial condition $\mu$ consisting of a suitable product of monochromatic distributions on distinct blocks, which achieves a better lower bound than the monochromatic distribution on a single block.
The analysis for this product measure is based on the detailed knowledge of the monochromatic case 
gathered within the proof of Theorem \ref{monocase}.

The strategy of proof in the continuous time setting is similar, with the crucial difference that in the graphical construction mentioned above the deterministic binary tree must be replaced by a random binary tree $\cT_t$ encoding the  collision history.  %In particular, the random environment now is not only a realization of the configurations at the leaves, but also of the collision history represented by the rooted tree $\cT_t$ alluded to above. 
 The tree $\cT_t$ can be defined, recursively, as follows: at time zero there is only one node, the root; after an exponentially distributed random time with mean $1$, the root gives rise to two new particles; independently,  each newly created particle repeats the same random splitting, and so on. $\cT_t$ represents the resulting random tree at time $t$, and the distribution $S_t(\mu)$ is calculated as an average over the realization of $\cT_t$ of the distribution obtained at the root when each internal node is assigned the distribution given by the collision product of the distributions assigned to its two children,  and the leaves are assigned the distribution $\mu$, see Lemma \ref{lem:rep_ct}. We let $L(\cT_t)$ denote the set of leaves of  $\cT_t$ and write $|x|$ for the depth of a leaf $x\in L(\cT_t)$, which is defined as the number of  splittings  along the path from the root to $x$.  
An important role in the analysis of the continuous time dynamics is played by the process defined as 
 \begin{equation}
W_t =e^{t/2} \sum_{x\in L(\cT_t)}4^{-|x|}\,,\qquad t\geq 0.
   \end{equation}
Using size-biasing, and a spinal representation in the spirit of  \cite{chauvin1988kpp,kyprianou2004travelling,lyons1995conceptual}, the process $W_t$ will be shown to define a uniformly integrable martingale, converging to a positive random variable $W_\infty$ with mean $\bE[W_\infty]=1$. 
In a sense that will be made more precise by our results later on, one can say that %the simple binary exponential decay $2^{-t}$ which characterizes 
the order parameter $s=n2^{-t}$, governing the cutoff transition in the discrete time dynamics, must be replaced in the continuous time setting by the random variable
\[
ne^{-t/2}W_t\,\longrightarrow\, r W_\infty.
\] 
In particular, the random variable $W$ appearing in the limiting profile function in Theorem \ref{monocasect} is precisely $W_\infty$.

As anticipated, there are no conceptual difficulties in extending our cutoff result to more general settings where balanced distributions $\mu\in\cB_n$ are replaced by arbitrary probability distributions $\mu$ on $\O$, provided that one assumes that the family of measures $\mu$ considered is such that the marginals $\mu(\si_i=1)$ are uniformly bounded away from $0$ and $1$. We give some more detailed comments in Section \ref{rem:extensions} below. Moreover, we believe that a cutoff phenomenon should continue to hold even for more general state spaces, that is when $\O=\{-1,+1\}^n$ is replaced by $\bbX^n$, with $\bbX$ a generic  finite alphabet, and all marginals are assumed to be uniformly nondegenerate.  

Finally, let's address some open questions that naturally emerge from the preceding discussion.
A first question concerns the existence of a cutoff phenomenon for nonlinear recombinations where the uniform choice of $A$ in the single recombination step \eqref{onestep} is replaced by a more general distribution $\nu$ over subsets of $[n]$. The speed of convergence in such models can be related to the probability, under $\nu$,  of separating any two given sites at each step, which gives the natural analog of \eqref{tvdistbound} and \eqref{ctconv2}, see \cite{Sinetal2} for the discrete time, and \cite{CapPar} for the continuous time.   We believe that a cutoff phenomenon should occur for such cases as well, but our proofs make explicit use of the uniformity in \eqref{onestep}. A second problem concerns the analysis of more general quadratic systems with nontrivial interactions in the stationary distribution, see \cite{CapSin,Sinetal} for  the definition of a general framework for quadratic reversible dynamical systems in combinatorial settings. A particular example is the nonlinear dynamics for Ising systems, obtained by a  modification of the one step recombination \eqref{onestep} where the content swap on the set $A$ is accepted or rejected, in such a way that the dynamics converges to the Ising distribution on a given interaction graph. This may be seen as a nonlinear version of the usual Metropolis algorithm.
In the high-temperature regime, tight bounds on the speed of convergence  for these models were recently derived \cite{CapSin2}. Drawing an analogy with the known cutoff results for high-temperature spin systems \cite{lubetzky2017universality}, we conjecture a similar cutoff phenomenon in the nonlinear version as well.

\section{Discrete time}
We start with a representation of the distribution $\mu_t=T_t(\mu)$ at time $t\in\bbN$, valid for any $\mu\in\cP(\O)$. This representation is convenient since it provides, conditionally given some random environment $\xi$ to be defined below, an independence between the $n$ bits.

\subsection{Graphical construction and fragmentation}\label{sec:graph_repr}
By definition, the solution $\mu_t=T_t(\mu)$, of the dynamical system \eqref{recstep} admits the following simple graphical interpretation. Consider the rooted regular binary tree with depth $t$, and assign to each node $u$ a distribution $\nu_u\in\cP(\O)$ in such a way that each of the $2^t$ leaves is assigned the distribution $\mu$ and  the distribution at each internal node $u$ is computed as the collision product of the two distributions at the children of $u$. Then $\mu_t$ is the distribution at the root of the tree,  see Figure \ref{fig:fig1} for the case $t=2$.
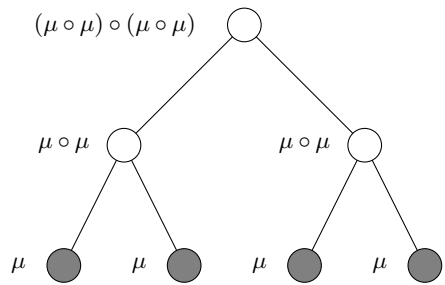
\begin{figure}[h]\label{fig:fig1}
\center
\begin{tikzpicture}[scale=0.8]
        \draw  (1,-1) -- (2,1);
\draw  (3,-1) -- (2,1);
   \draw  (2,1) -- (4,3);
    \draw  (5,-1) -- (6,1);
    \draw  (7,-1) -- (6,1);
    \draw  (6,1) -- (4,3);
    \node[shape=circle, draw=black, fill = white, scale = 1.5]  at (4,3) {}; 
    \node[shape=circle, draw=black, fill = white, scale = 1.5]  at (2,1) {};
    \node[shape=circle, draw=black, fill = gray, scale = 1.5]  at (1,-1) {};
    \node[shape=circle, draw=black, fill = gray, scale = 1.5]  at (3,-1) {};
    \node[shape=circle, draw=black, fill = white, scale = 1.5]  at (6,1) {};
    \node[shape=circle, draw=black, fill = gray, scale = 1.5]  at (5,-1) {};
    \node[shape=circle, draw=black, fill = gray, scale = 1.5]  at (7,-1) {};
        \node at (1/4,-1) {$\mu$};
\node at (9/4,-1) {$\mu$};
    \node at (17/4,-1) {$\mu$};
    \node at (25/4,-1) {$\mu$};
    \node at (1,1) {$\mu \circ \mu$};
    \node at (20/4,1) {$\mu \circ \mu$};
    \node at (7/4+.1,3) {$(\mu\circ \mu) \circ (\mu \circ \mu)$};
\end{tikzpicture}
\caption{$T_2(\mu)$ is the distribution at the root of the binary tree with depth $2$.}
\end{figure}

It is useful to have a pathwise implementation of this graphical construction. 
Set $N:=2^t$, and let $\xi(1),\dots,\xi(N)$ denote $N$ independent random variables with law $\xi(x)\sim \mu$ for every leaf $x\in\{1,\ldots,N\}$. Define
\begin{equation}\label{eq:qxii}
q^\xi(i):= \frac1{N}\sum_{x=1}^{N} \xi_i(x),\qquad i\in[n],
\end{equation}
where $\xi_i(x)\in\{-1,+1\}$ denotes the $i$-th bit in the string $\xi(x)\in\O$. 
%Recall that $\pi$ is the uniform distribution over $\O$. 
The next lemma expresses $\mu_t$ as an average $\bbE[\mu_t^\xi]$, where $\bbE$ denotes the expectation w.r.t.\ the random variables $\xi=\{\xi(x),\, x=1,\dots,N\}$, and the quenched measure $\mu_t^\xi$ is a product of $n$ Bernoulli random variables each with marginal 
\[
\mu_t^\xi(\si_i=\pm 1) = (1\pm q^\xi(i))/2.
\] 
\begin{lemma}\label{lem:rep}
For any $\mu\in\cP(\O)$, we have $\mu_t = \bbE[\mu_t^\xi]$, where 
\begin{equation}\label{eq:prod}
\mu^{\xi}_t(\si)=2^{-n}\prod_{i=1}^n (1+\si_i q^{\xi}(i)).
\end{equation}
\end{lemma}
\begin{proof}
We start by observing that for every $t\in\bbN$, $\mu\in\cP(\O)$, %Our starting point is the following representation for the measure $\mu_t$. Recalling that $N=2^{t}$, we have
\begin{equation}\label{identity}
\mu_t=\frac{1}{N^n}\sumtwo{(A_1,\dots, A_{N}):}{\sqcup A_i=[n]} \mu_{A_1}\otimes \mu_{A_2}\otimes \dots \otimes \mu_{A_{N}},
\end{equation}
where the sum is taken over all partitions of $[n]$ into $N=2^t$ disjoint  sets (which are allowed to be empty).
The case $t=1$ is just the definition of uniform recombination \eqref{onestep} and the case $t\ge 2$ follows by induction since if $(A_1,\dots, A_{N})$, $(A'_1,\dots, A'_{N})$ are two independent uniformly random partitions of $[n]$ into $N$ disjoint sets, and $A\subset [n]$ is taken independently and uniformly at random, then 
\begin{equation}\label{eq:fragsets}
(A_1\cap  A,\dots, A_{N}\cap A, A'_1\cap A^c,\dots, A'_{N}\cap A^c)
\end{equation}
yields  a uniformly random partition of $[n]$ into $2N$ disjoint sets.

A way  to sample such a partition uniformly at random amongst the $N^{n}$ possibilities is to assign  to each $i\in [n]$ independently a uniformly random number $U_i$ in $[N]=\{1,\dots,N\}$ and define
\begin{equation}\label{eq:setsai}
A_x:= \{ i\in [n] \ : U_i=x\},\qquad x=1,\dots,N.
\end{equation}
We denote the associated probability by $P$, and write $E$ for the corresponding expectation.
With this definition, \eqref{identity} becomes
\begin{equation}
 \mu_t= E\left[\mu_{A_1}\otimes \mu_{A_2}\otimes \dots \otimes \mu_{A_{N}} \right].
\end{equation}
Thus, if $\xi(x)=(\xi_1(x),\dots,\xi_n(x))$, $x\in[N]$, are IID with distribution $\mu$, then for any $\si\in\O$, using Fubini Theorem
\begin{equation}\begin{split}
  \mu_t(\sigma)&= \textstyle{E\left[\prod_{x=1}^N\mu_{A_x}(\sigma_{A_x})  \right]}\\
  &= E\left[ \bbP\left[\forall x \in [N],\  \forall i\in A_x,\  \xi_{i}(x)=\sigma_i \right]\right]\\
  &= \bbE\left[ P \left[ \forall x \in [N],\  \forall i\in A_x,\ \xi_{i}(x)=\sigma_i\right]\right]\\
  &=  \bbE\left[ P \left[ \forall i\in [n],  \ \xi_i(U_i)=\sigma_i\right]\right].
\end{split}
  \end{equation}
  Therefore, $\mu_t=\bbE[\mu^\xi_t]$, %To conclude we just need to observe that 
 where the quenched probability  $\mu^{\xi}_t$ coincides with the distribution of $(\xi_i(U_i))^n_{i=1}$ given $\xi$. Since, conditionally on $\xi$,  the $\xi_i(U_i)$, $i\in[n]$, are independent and $\xi_i(U_i) = \pm1$ with probability $(1\pm q^\xi(i))/2$, this proves \eqref{eq:prod}.
\end{proof}
The above lemma can be summarized as follows. If each leaf $x\in[N]$ is assigned the configuration $\xi(x)$, where the $(\xi(x))$ are IID with law $\mu$ and each $i\in[n]$ is assigned the leaf $U_i$,  where the $(U_i)$ are IID uniform in $[N]$, then the configuration at the root  given by $\si=(\xi_1(U_1),\dots,\xi_n(U_n))$ has law $\mu_t$. Conditionally on the variables $(U_i)$ the vector $\si$ has the law $\mu_{A_1}\otimes \dots \otimes \mu_{A_{N}}$, where the subsets $A_x$ are given in \eqref{eq:setsai}.  The partition $(A_x)$ can be seen as the result of a fragmentation of the set $[n]$ at time $t$. More precisely, we define a discrete time fragmentation process $\cA_t$ such that $\cA_0 = [n]$, and if $\cA_t = (A_x)^{2^t}_{x=1}$, then $\cA_{t+1}=(A'_y)^{2^{t+1}}_{y=1}$ is obtained by replacing each  $A_x$ by the two sets $A_x\cap A$ and $A_x\cap A^c$ , where $A$ is a uniformly random subset of $[n]$. Note that at each time $t$ this gives a uniformly random partition of $[n]$ into $N=2^t$ subsets. We may define the  {\em fragmentation time} $\t_{\rm frag}$ as the first time $t$ such that all subsets in $\cA_t$ are either empty or a singleton. This coincides with the first time such that for all indexes $1\le i<\ell\le n$, one has  $i$ and $\ell$ in different subsets of the partition. 
By the above construction one sees that conditionally on $\{\t_{\rm frag}\le t\}$, the random vector $\si$ at the root at time $t$ is a product of independent bits, and is thus distributed according to $\pi_\mu$. It follows that
   \begin{equation}\label{eq:ub0}
 \|\mu_t-\pi_\mu\|_{\rm TV}\le \bbP(\t_{\rm frag}> t)\le \tfrac12n(n-1)%\binom{n}2\,
 2^{-t}\,,
 \end{equation}
 where we use a union bound over pairs $1\le i<\ell\le n$ and the fact that for a given such pair, the probability to have $i$ and $\ell$ in the same set of the partition $\cA_t$ at time $t$ is exactly $2^{-t}$.
%Since the $U_i$ %$J(\zeta,i)$, $i=1,\dots,n$ 
%are IID uniformly distributed over the leaves $\{1,\dots,N\}$, 
Moreover, it is not hard to see that $\t_{\rm frag}$ is actually concentrated around $2\log_2n$, that is the function $t\mapsto \bbP(\t_{\rm frag}> t)$ has a cutoff behavior at $t=2\log_2n$ with an $O(1)$ window. 

\subsection{Upper bound for discrete time}

We provide now a better strategy, which allows us to prove the upper bound \eqref{ubound} on the TV distance. This improves the bound given in \eqref{eq:ub0}  by a factor $(n-1)/2$, thus achieving the optimal bound $\log_2n+O(1)$ for the cutoff time. The next lemma establishes the upper bound in Theorem \ref{th:discrete}.

\begin{lemma}\label{lem:ubounds}
For every $n,t\in\bbN$,
for all $\mu\in\cB_n$,
 \begin{gather}
 \|\mu_t-\pi\|_{\rm TV}\le n 2^{-t}\;.
 \label{eq:rootbd}
 \end{gather}

 \end{lemma}

\begin{proof}
From Lemma \ref{lem:rep} we have $\mu_t=\bbE[\mu_t^\xi]$, where $\mu_t^\xi\in\cP(\O)$ has density 
\begin{equation}
h^\xi_t(\si)=\frac{\dd\mu^{\xi}_t}{\dd\pi}(\si)=\prod_{i=1}^n (1+\si_iq^{\xi}(i)).
\end{equation}
Using $\bbE\left[q^{\xi}(i)\right]=0$ for all $i\in[n]$, we obtain
\begin{align}
  \| \mu_t-\pi \|_{\rm  TV}& =\tfrac12\,\| \bbE\, (h^\xi_t-1) \|_{L^1(\pi)}\nonumber 
  \\&= \tfrac12\,\textstyle{\| \bbE\, (h^\xi_t-\sum^n_{i=1}\si_i q^{\xi}(i)-1) \|_{L^1(\pi)}}\nonumber 
  \\ & \leq \tfrac12\,\textstyle{\bbE\, \|  h^\xi_t-\sum^n_{i=1}\si_i q^{\xi}(i)-1 \|_{L^1(\pi)}}
  \label{ubidea11}\,.
\end{align}
On the one hand, using at the last line that $\si_i$ are IID centered r.v.~under $\pi$
\begin{align}
	&\textstyle{\|  h^\xi_t-\sum^n_{i=1}\si_i q^{\xi}(i)-1 \|_{L^1(\pi)}}\\
	&\qquad \leq \int h^\xi_t d\pi + 1 + \int \textstyle{\left|  \sum^n_{i=1}\si_i q^{\xi}(i) \right| d\pi}\\
	&\qquad \leq 2+ \left(\int\textstyle{\left(\sum_{i=1}^n \si_i q^{\xi}(i)\right)^2\dd }\pi  \right)^{1/2}
	%\\
	= %&\qquad \leq 
	2 +  \textstyle{\langle q^\xi,q^\xi\rangle^{1/2}}\;,
\label{ubidea101}
\end{align}
where we use the notation $\langle x,y\rangle=\sum_{i=1}^nx(i)y(i)$. On the other hand,
writing $\hat h^\xi_t=h^\xi_t-\sum^n_{i=1}\si_i q^{\xi}(i)$, we estimate
\begin{align}
		\textstyle{\|  h^\xi_t-\sum^n_{i=1}\si_i q^{\xi}(i)-1 \|^2_{L^1(\pi)}} &\leq\| \hat h^\xi_t-1 \|^2_{L^2(\pi)}
=\int\left(\hat h^\xi_t\right)^2\!\dd\pi \,-\,1\\
		&= \prod_{i=1}^n \left(1+q^{\xi}(i)^2\right) - \sum^n_{i=1} q^{\xi}(i)^2-1\\
		&\le e^{\langle q^{\xi},q^{\xi}\rangle }- \langle q^{\xi},q^{\xi}\rangle -1\;.
\label{eq:ubao1}
\end{align}
Using the inequality
\begin{equation}
	\min\left\{ \left(e^{\langle q^{\xi},q^{\xi}\rangle}-\langle q^{\xi},q^{\xi}\rangle-1\right)^{1/2},2+\langle q^{\xi},q^{\xi}\rangle^{1/2} \right\}
	\le 2 \langle q^{\xi},q^{\xi}\rangle\;,
\end{equation}
we have thus shown that
\begin{equation}\label{eq:qxib}
  \| \mu_t-\pi \|_{\rm  TV} \le \bbE\left[\langle q^{\xi},q^{\xi}\rangle\right] = n \bbE\left[(q^\xi(1))^2\right] = n2^{-t}\,
  .\end{equation}
\end{proof}

The previous result establishes an upper bound on $D_n(t)$ which, for large $t$ or equivalently for small $s=n2^{-t}$, is asymptotically sharp up to constant (cf.\ the discussion in Section \ref{ideas} or Section \ref{hellingersection} below). %the discussion in Section \ref{ideas}). 
However it does not yield any information about the case $t$ small or equivalently $s$ large.
The following lemma fills this gap by providing a bound, which captures the decay of the distance to equilibrium at times $t=\log_2n - K$ for large constants $K$. It is in a sense quite close to the asymptotic lower bound \eqref{lbound} which we will obtain in Section \ref{hellingersection}, as the two only differ by the constant that appears in the exponential.

\begin{lemma}\label{lem:ubounds2}
For every $n,t\in\bbN$, setting $s:= n 2^{-t}$, and for all $\mu\in\cB_n$, 
\begin{equation}
 \label{unilow}
 \|\mu_t-\pi\|_{\rm TV}\le 1- \tfrac12\,e^{-2s}
 %\frac{e^{-2s}}{2}
 ,\qquad s\ge \log(2)/2.
 \end{equation}
 \end{lemma}
\begin{proof}
We proceed as in the previous proof but instead of using only Schwarz' inequality to bound the $L^1(\pi)$ norm by the $L^2(\pi)$ norm, we use the following finer inequality, whose proof is given in Appendix \ref{app:tricky}, namely
\begin{equation}\label{eq:finerineq}
\tfrac12 \,  \|f - 1\|_{L^1(\pi)} \le \phi \left(\|f - 1\|_{L^2(\pi)} \right),
\end{equation}
with
$$ \phi(x)=\begin{cases} 
x/2,  &\text{ if } x\le 1,\\
         \frac{x^2}{1+x^2} \quad &\text{ if } x\ge1.
        \end{cases}$$
Note that
\begin{align}
&\| h^\xi_t-1 \|^2_{L^2(\pi)}  =\int\left( h^\xi_t\right)^2\!\dd\pi \,-\,1 = \prod_{i=1}^n \left(1+q^{\xi}(i)^2\right) -1
  \le e^{\langle q^{\xi},q^{\xi}\rangle } -1.
\end{align}
Using monotonicity of $\phi$ and \eqref{eq:finerineq}, %we have
\begin{equation}
  \|\mu_t-\pi\|_\tv \le \bbE\left[ \phi\left( \sqrt{e^{\langle q^{\xi},q^{\xi}\rangle}-1}\,\right)\right].
  %\le \frac{1}{2} \left( \phi(\sqrt{e^{2m_t}-1})+1\right)=1-\frac{e^{-2m_t}}{2}.
\end{equation}
Moreover,
 $\bbE\left[\langle q^{\xi},q^{\xi}\rangle \right]= n 2^{-t}=s$,
 and therefore,  by Markov's inequality,
 \[
 \bbP\left(\langle q^{\xi},q^{\xi}\rangle\ge 2s \right)\le \frac12\,.
 \] 
  Since $\phi$ is non-decreasing and bounded by $1$, and provided $2s\ge \log 2$, we obtain
\begin{align*}
  \|\mu_t-\pi\|_\tv &\le \bbE\left[\mathbf{1}_{\langle q^{\xi},q^{\xi}\rangle\ge 2s} + (1-\mathbf{1}_{\langle q^{\xi},q^{\xi}\rangle\ge 2s}) \phi\left(\sqrt{e^{2s}-1}\,\right)\right]\\
  &\le \bbP\big(\langle q^{\xi},q^{\xi}\rangle\ge 2s \big) + \big(1-\bbP\big(\langle q^{\xi},q^{\xi}\rangle\ge 2s \big)\big) (1-e^{-2s})\\
  &\le  1- \tfrac12\,e^{-2s}.
\end{align*}

\end{proof}

\subsection{Explicit profile for monochromatic initial states}\label{sec:monodisc}
Here we prove Theorem \ref{monocase}. We take the monochromatic distribution $ \mu =  \tfrac12\,\ind_{-} + \tfrac12\,\ind_{+}$ as initial state. Using the notation from Lemma \ref{lem:rep}, %introduced above, 
we note that now $q^{\xi}(i)$ does not depend on $i$. We let $\bar q^{\xi}$ denote this common value, and write $ \rho_t$ for the probability density
\begin{equation}
 \rho_t(\si)=\frac{\dd \mu_t}{\dd \pi}(\si)=\bbE\left[h^\xi_t\right] = \bbE\left[ \prod_{i=1}^n (1+\si_i \bar q^{\xi})\right].
\end{equation}
\begin{proposition}\label{maincase}
 If $t_n\in\bbN$ is such that  $\lim_{n\to \infty} n2^{-t_n}=s$ for some $s>0$, then for any bounded continuous function 
 $F:\bbR_+ \to\bbR$,
\begin{equation} \label{inlaw}
\lim_{n\to \infty}  \int F(\rho_{t_n}(\si))\pi(\dd \si)= \int \frac{e^{-\frac{z^2}{2}}}{\sqrt{2\pi}}\,F\left( \gamma_s(z) \right)  \dd z,
\end{equation}
with
\begin{equation} \label{eq:gamma_s}
\gamma_s(z):=\frac{e^{\frac{sz^2}{2(s+1)}}}{\sqrt{s+1}}\;,\quad z\in\R\;.
\end{equation}
\end{proposition}
The asserted convergence can be interpreted as convergence in distribution of $\rho_{t_n}$ viewed as a r.v.~on $(\Omega,\pi)$ towards $\gamma_s$ viewed as a r.v.~on $(\R,\cN(0,1))$. As a consequence, we can restrict ourselves to smoother functions $F$ in the proof (we will take Lipschitz bounded functions).
Let us first observe that Proposition \ref{maincase} implies Theorem \ref{monocase}.
\begin{proof}[Proof of Theorem \ref{monocase}]
By definition \eqref{eq:tvnormal} and since $\gamma_s$ is the relative density of $\cN(0,1+s)$ with respect to $\cN(0,1)$, one has
\[
\varphi(s)=\|\cN(0,1+s)- \cN(0,1) \|_{\tv}=\frac 1 2 \int \frac{e^{-\frac{z^2}{2}}}{\sqrt{2\pi}}\left|\gamma_s(z)-1\right| \dd z.
\]
Thus, we need to prove that for $t_n\in\bbN$ such that  $\lim_{n\to \infty} n2^{-t_n}=s$,
\begin{equation}\begin{split}\label{conseq}
 \lim_{n\to \infty}
 \int \left|\rho_{t_n}(\si)-1\right|\pi(\dd \si)
=\int \frac{e^{-\frac{z^2}{2}}}{\sqrt{2\pi}}\left|\gamma_s(z)-1\right| \dd z.
 \end{split}\end{equation}
Since the map $F:x\mapsto 1+x - |x-1|$ is bounded and continuous on $\R_+$, Proposition \ref{maincase} yields
\begin{align}
&\lim_{n\to \infty}\int \left[(1+\rho_{t_n}(\si))-|\rho_{t_n}(\si)-1|\right]\pi(\dd \si) \\
&\qquad \qquad = \int \frac{e^{-\frac{z^2}{2}} }{\sqrt{2\pi}}\left[ \left(
1+\gamma_s(z)\right) - \left|\gamma_s(z)-1\right|\right]\dd z.
\end{align}
Since $\int (1+\rho_{t_n}(\si)) \pi(\dd \si)=2= \int\frac{e^{-\frac{z^2}{2}} }{\sqrt{2\pi}}\left(1+\gamma_s(z)\right)\dd z$, \eqref{conseq} follows.
\end{proof}

\begin{proof}[Proof of Proposition \ref{maincase}]
Without loss of generality, we can assume $F$ Lipschitz and bounded. To prove \eqref{inlaw}
we are going to show that  $\rho_{t_n}(\si)$ is well approximated by
\begin{equation}\label{need1}
\gamma_s(\bar \si_n) = \frac{e^{\frac{s\,\bar \si_n^2}{2(s+1)}}}
{\sqrt{s+1}},\qquad \bar \si_n: = \frac1{\sqrt n}\sum_{i=1}^n \si_i\,.
\end{equation}
%where we introduce the notation $\bar \si_n: = \frac1{\sqrt n}\sum_{i=1}^n \si_i$. 
That is, we show that %is to say that
\begin{equation}\label{need2}
 \lim_{n\to \infty} \int \left|F(\rho_{t_n}(\si))-F(\gamma_s(\bar \si_n))\right|\pi(\dd \si)=0.
 \end{equation}
Once \eqref{need2} is available,  it is sufficient to prove \eqref{inlaw} with $\rho_{t_n}(\si)$ replaced by $ \gamma_s(\bar \si_n)$, that is 
\begin{equation} \label{inlawbar}
\lim_{n\to \infty}  \int F( \gamma_s(\bar \si_n))\pi(\dd \si)= \int \frac{e^{-\frac{z^2}{2}}}{\sqrt{2\pi}}\,F\left(  \gamma_s(z)\right)  \dd z. 
\end{equation}
However, since $F\circ \gamma_s$ is bounded and continuous, the above convergence follows simply from the fact that  $\bar \si_n$ converges in distribution, under $\pi$, to a standard Gaussian.
We are left with the proof of \eqref{need2}.
Observing that
\[
\prod_{i=1}^n (1+\si_i \bar q^{\xi})^2 = (1+\bar q^{\xi})^{\sqrt n{ \bar \si_n}}
(1-\bar q^{\xi})^{-\sqrt n{ \bar \si_n}} \left( 1-(\bar q^{\xi})^2 \right) ^n,
\]
one has %We have 
\begin{equation}\label{distrybe1}
 \rho_{t_n}(\si)=\bbE\left[ e^{\alpha_n^{\xi} \bar \si_n+ \beta_n^{\xi}}\right],
\end{equation}
where
\begin{equation}\label{degalfbet}
 \alpha_n^{\xi}:=\frac{\sqrt{n}}{2}\log \left(\frac{1+\bar q^{\xi}}{1-\bar q^{\xi}} \right), \quad \beta_n^{\xi}:=\frac{n}{2}\log \left( 1-(\bar q^{\xi})^2 \right) ,
 %\quad \text{ and } \quad  \bar \si_n:=n^{-1/2}\sum^n_{i=1}\si_i
\end{equation}
with appropriate convention for the special case $\bar q^{\xi}=1$.
The key observation for the proof is the following joint convergence in distribution
\begin{equation}\label{distrybe}
 \lim_{n\to \infty}(\alpha^{\xi}_n,\beta^{\xi}_n)\stackrel{(d)}{=} \left(Z_s,-\frac{1}{2}Z_s^2\right),
\end{equation}
where $Z_s\sim \cN(0,s)$. Indeed, \eqref{distrybe} is a direct consequence of an expansion of the logarithm in \eqref{degalfbet}, using $n2^{-t_n}\to s$ and the convergence in distribution of
$$ \sqrt{2^{t_n}}\bar q^\xi = \frac1{\sqrt{2^{t_n}}} \sum_{x=1}^{2^{t_n}} \xi_1(x)$$
to a standard gaussian.

If  in \eqref{distrybe1} we replace formally the variables $\alpha_n^{\xi} ,\beta_n^{\xi}$ by their limits, then the computation of a gaussian integral gives the desired result, namely that $ \rho_{t_n}$  is well approximated by
\eqref{need1}.   At this point, to conclude the proof there are two issues one has to consider: (a) we need more than convergence in distribution to replace $\alpha^{\xi}_n$ and $\beta^{\xi}_n$ in the expectation, (b) the fact that the space in which $\si$ lives depends on $n$ makes passing to the  limit more tricky.
Given $a\in \bbR$, we set  
$$g_n(a,\bar q^\xi):=  e^{\alpha_n^{\xi} a+ \beta_n^{\xi}}.$$
Note that $\alpha_n^{\xi}$ and $\beta_n^{\xi}$ only depend on $\bar q^\xi$ so that this definition makes sense. Note the dependence on $n$ is also hidden in the fact that $\xi$ depends on $t=t_n$.
An important observation is the following
\begin{lemma}\label{unifconv}
If  $\lim_{n\to \infty} n2^{-t_n}=s$, then for all $a\in\bbR$,
\begin{equation}\label{simplconv}
 \lim_{n\to \infty}\bbE\left[ g_n(a,\bar q^\xi) \right] =\frac 1{\sqrt{s+1}}e^{\frac{sa^2}{2(s+1)}} = \gamma_s(a), 
\end{equation}
and the convergence is uniform on the interval $a\in[-A,A]$ for any $A>0$.
\end{lemma}

The proof of Lemma \ref{unifconv} will be given below. We can now conclude the proof of \eqref{need2}. To that end, it suffices to show that for any given $A>0$
\begin{equation}\label{Eq:monotoprove1}
	\limsup_{n\to\infty} \int  \Big\vert F(\rho_{t_n}(\si)) - F(\gamma_s(\bar \si_n)) \Big\vert \mathbf{1}_{\{\vert \bar \si_n\vert \le A\}}\pi(\dd \si) = 0\;,
\end{equation}
and that
\begin{equation}\label{Eq:monotoprove2}
	\lim_{A\to\infty} \limsup_{n\to\infty} \int  \Big\vert F(\rho_{t_n}(\si)) -  F(\gamma_s(\bar \si_n)) \Big\vert \mathbf{1}_{\{\vert \bar \si_n\vert > A\}}\pi(\dd \si) = 0\;.
\end{equation}
We start with the latter. By Hoeffding's inequality
$$ \pi(\{\vert \bar \si_n\vert > A\}) \le 2 e^{-A^2/2}\;,\quad A>0\;.$$
Since $F$ is bounded, we easily deduce \eqref{Eq:monotoprove2}.
On the other hand, since $\rho_{t_n}(\si)$ coincides with $\bbE\left[ g_n(a,\xi) \right]$ evaluated at $a=\bar\si_n$, \eqref{Eq:monotoprove1} follows from Lemma \ref{unifconv} and the fact that $F$ is Lipschitz.
\end{proof}

\begin{proof}[Proof of Lemma \ref{unifconv}]
We claim that for any $A>0$, % it holds
\begin{equation}\label{eq:maxb}
 \sup_{|a|\le A}\limsup_{n\ge 1}\max\left\{ \|g_n(a,\cdot) \|_{\infty}, \|\partial_a g_n(a,\cdot) \|_{\infty} \right\}<\infty.
 \end{equation}
Indeed, computing the derivative in $\bar q^\xi$ of $g_n(a,\bar q^\xi)$ shows that it is maximized at $\bar q^\xi =a/\sqrt{n}$ and from this, we deduce the bound on $\|g_n(a,\cdot) \|_{\infty}$. Regarding the bound on $\partial_a g_n(a,\cdot)$, we observe that $|\partial_a g_n(a,\bar q^\xi)|\le \max\{ g_n(a+1,\bar q^\xi) ; g_n(a-1,\bar q^\xi)\}$.\\

As a consequence of the bound \eqref{eq:maxb}, we deduce that for $a\in [-A,A]$, the map $g_n(a,\cdot)$ coincides with $g_n(a,\cdot) \wedge M$ for some positive constant $M$. The latter is a continuous bounded function of $\alpha_n^{\xi}$ and $\beta_n^{\xi}$, so the convergence in law \eqref{distrybe} implies that 
\begin{equation}\label{Eq:CVgn}
\lim_{n\to \infty}\bbE\left[ g_n(a,\bar q^\xi) \right] = \bbE\left[ e^{aZ_s-\frac{1}{2}Z_s^2}\right]=\frac1{\sqrt{s+1}}e^{\frac{s\,a^2}{2(s+1)}}.
\end{equation}
The bound on the derivative in $a$ of $g_n(a,\bar q^\xi)$ proven above suffices to deduce that $a\mapsto \bbE\left[ g_n(a,\bar q^\xi) \right]$ is equicontinuous on $[-A,A]$. From this property, we deduce that the convergence \eqref{Eq:CVgn} holds uniformly. 
\end{proof}

We conclude with some remarks on the %that it is possible to obtain 
asymptotic behavior  in the regime where $n2^{-t}\to\infty$. For instance, when $t\in\bbN$ is fixed one has the following behavior.

\begin{lemma}\label{sidecase1}
 If $t\in\bbN$ is fixed, and $\mu=\frac{1}{2}\ind_+ +\frac{1}{2}\ind_-$, then
  \begin{equation}\label{eq:nomo}
   \lim_{n\to \infty} \|\mu_{t}-\pi\|_\tv= 1 - \binom{2^t}{2^{t-1}}\,2^{-2^t}.
\end{equation}
\end{lemma}
\begin{proof}
From Lemma \ref{lem:rep},
\begin{equation}
  \|\mu_{t}-\pi\|_\tv\leq  \bbE\left[\big\| \mu^{\xi}_{t}-\pi\big\|_\tv\right].
\end{equation}
Since we are in the monochromatic case, % the $\xi_i$ are either uniformly $+$ or uniformly $-$ and thus, 
under $\mu^{\xi}_t$, the $(\sigma_i)$ are IID $\{-1,1\}$ valued Bernoulli variables with parameter $\frac{1+\bar q^{\xi}}{2}$. In particular $\mu^{\xi}_{t}=\pi$ if $\bar q^{\xi}=0$. Hence we have
 \begin{equation}
   \bbE\left[ \big\|\mu^{\xi}_{t}-\pi\big\|_\tv\right]
   \le  \bbP[\bar q^{\xi}\ne 0]=1  -\binom{2^t}{2^{t-1}}\,2^{-2^t}.
 \end{equation}
In order to prove a matching lower bound,  consider the event
$$A_t:= \textstyle{\left\{ \left|\sum^n_{i=1}\sigma_i\right|\ge n2^{-t} \right\}}.$$
Then, \begin{align}
 \|\mu_{t}-\pi\|_\tv\ge \mu_{t}(A_t)-\pi(A_t)
 &= \bbE\left[\mu^{\xi}_t(A_t)\right]-\pi(A_t)\\
 &\ge \bbE\left[\mathbf{1}_{\{\bar q^{\xi}\ne 0\}} (1-\mu^{\xi}_t(A_t^c))\right] - \pi(A_t)\;.
 \end{align}
 By Hoeffding's inequality, we know that $\pi(A_t)\le 2 e^{-n 2^{-2t}/2}$. On the other hand, note that $\bar q_\xi$ takes its values in the set $\{ 2k 2^{-t}: k =-2^{t-1},\ldots, 2^{t-1}\}$. Consequently, assuming $\bar q^{\xi}\ne 0$, we have $|\bar q^{\xi}|\ge 2^{1-t}$ and therefore Hoeffding's inequality yields
 \begin{equation}
  \mu^{\xi}_t(A_t^c) \le \mu^{\xi}_t\left( |\textstyle{ \sum^n_{i=1}}\sigma_i - n \bar q^\xi | \ge n 2^{-t}\right) \le 2 e^{-n 2^{-2t} / 2}\;.
 \end{equation}
Hence we have
\begin{equation}
 \|\mu_{t}-\pi\|_\tv\ge \bbP[\bar q^{\xi}\ne 0] - 4e^{-n 2^{-2t}/2},
\end{equation}
which concludes our proof.
\end{proof}
\begin{remark}\label{rem:monotone}
Lemma \ref{sidecase1} shows in particular that the distance to equilibrium is far from being monotone for monochromatic initial states, see Figure \ref{fig:mon}. Indeed, in the limit of $n$ large, it drops from $1$ to $1/2$ in one step, it is \textit{macroscopically increasing} during the first few subsequent steps and then 
it gradually increases to $1$ before entering the cutoff window $t_{n,\lambda}=\lfloor\log_2n +\lambda\rfloor$, $\lambda\in\bbR$, where it settles to the cutoff profile as stated in Theorem \ref{monocase}.
%Unlike our other results, this one is specific to the particular choice we have made for the marginal of $\mu$.
  \end{remark}
%\bigskip
\begin{figure}[h]
	\centering
	\includegraphics[width=5cm]{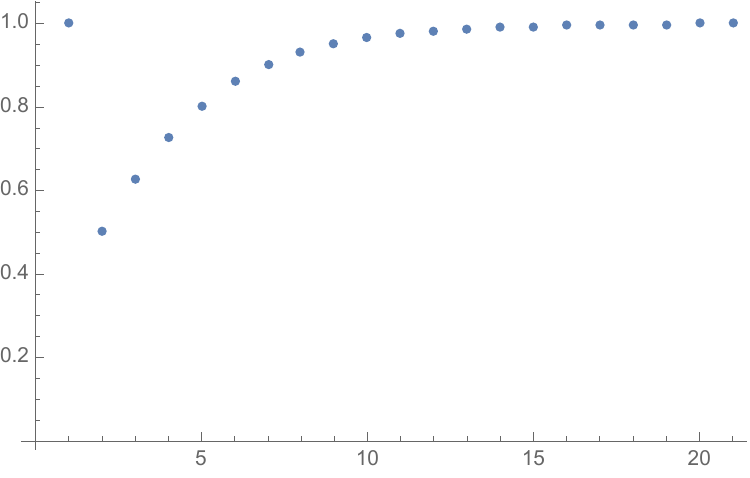}\\
	\caption{Plot of the function in the right hand side of \eqref{eq:nomo} expressing the total variation distance  after $t$ steps, started from the monochromatic distribution, in the limit $n\to\infty$.
	%, see Lemma \ref{sidecase1}
	}	\label{fig:mon}
\end{figure}

\begin{remark}\label{rem:extra}
We observe that, with minor modifications, the argument in Lemma~\ref{sidecase1} can be extended to obtain the following statement:  
if $1 \ll 2^t\ll \sqrt{n}$, then 
\begin{equation}
    \|\mu_{t}-\pi\|_\tv= 1-  \frac{1}{\sqrt{\pi 2^{t}}}(1+o(1)). %o(2^{-t/2}).
\end{equation}
On the other hand, with a bit of extra work, the argument in the  proof of Theorem \ref{monocase}, combined with the asymptotics 
$\varphi(s)=1 - \frac{2\sqrt{\log s}}{\sqrt{2\pi s}}(1+o(1))$, as $s\to \infty$, see Appendix \ref{app:Gauss},
  can be seen to imply that
 if $\sqrt{n} \ll 2^t\ll n$, then  we have 
\begin{equation}
\|\mu_{t}-\pi\|_\tv= 1- \frac{2\sqrt{\log (n/2^{t})}}{\sqrt{2\pi (n/2^{t})}}(1+o(1))
 %O\(\(n2^{-t}\)^{-1/2}\) O\(\sqrt{2^t/n}\). %n^{-1/2}2^{t/2}). %. %+ o(n^{-1/2}2^{t/4}).
\end{equation}
\end{remark}

\subsection{Lower bound for discrete time}\label{hellingersection}
The lower bound obtained in Theorem \ref{monocase} for the monochromatic state is of order $s$ when $s$ is small. More precisely, as shown in Appendix \ref{app:Gauss}
\[
\varphi(s)= \frac{s}{\sqrt{2e\pi}}\,(1+o(1)),\qquad s\downarrow 0\,,
\]
which matches up to constant with the upper bound we found in Theorem \ref{th:discrete}.
However, for large values of $s$, 
\[
\varphi(s)=1 - \frac{2\sqrt{\log s}}{\sqrt{2\pi s}}\,(1+o(1))\,,\qquad s\to \infty\,,
\]
which is very different from \eqref{unilow}, see Appendix \ref{app:Gauss}. It turns out that non monochromatic initial condition obtained by taking products of monochromatic distributions can produce something much closer to \eqref{unilow}.

\begin{proposition}\label{prop:lowerbd}
 There exists $c>0$ such that for every $n\ge 1$ and $t\ge 1$, setting $s= n 2^{-t}$, we have
 \begin{equation}
  D_n(t)\ge 1- 2e^{-cs}\,.
 \end{equation}
\end{proposition}
\noindent Combining Lemma \ref{lem:ubounds} and Proposition \ref{prop:lowerbd}, we obtain Theorem \ref{th:discrete}.

\begin{proof}
First of all, assume that the result holds under the further restriction that $s\ge 80$. Then, one can diminish the value of $c$ in such a way that $1-2e^{-cs}$ is negative for all $s\in [0,80]$ so that the asserted bound trivially holds for these $s$. We can therefore assume that $s\ge 80$.\\
We split the set $[n]$ of coordinates into groups of cardinality $p:=80\times 2^{t}$ and, if $n$ is not a multiple of $p$, an additional group of cardinality strictly less than $p$. We let $\alpha= \lfloor s/80 \rfloor$ denote the number of group of cardinality $p$ that one obtains in this manner, note that $\alpha \ge 1$ since we work with $s\ge 80$. Thus, it is sufficient % We are going 
to find a measure $\mu$ and an event $A$ which are such that
$$ \pi(A)\le e^{-c\alpha} \;\;\text{ and } \;\;\mu_t(A^c)\le e^{-c\alpha}\,. $$
We let $\Xi_i$ denote the squared magnetization of the $i$-th group.
\begin{equation}
 \Xi_i:= \left(\sum_{j=(i-1)p+1}^{ip} \sigma_j\right)^2.
\end{equation}
We set
\begin{equation}
X_i:=\ind_{\{\Xi_i\ge 20p\}}, \quad  Z:=\sum_{i=1}^{\alpha}  X_i \quad \text{ and }
 A:=\left\{ Z \ge \frac{\alpha}{15} \right\}.
\end{equation}
Finally we set $\mu= (\nu_p)^{\otimes \alpha}\otimes \nu_{n-p\alpha}$, where $\nu_k$ denotes the monochromatic measure $\frac{1}{2}(\ind_{+}+\ind_-)$ on $\{-1,1\}^k$. That is to say that under $\mu$ the spin configuration is monochromatic on each block of size $p$, but the spins given to blocks are IID.
Let us first estimate  $\pi(A)$.
We have by Markov's inequality
\begin{equation}
 \pi(\Xi_1\ge 20p)\le \frac{\pi(\Xi_1)}{20p}=\frac{1}{20}.
\end{equation}
(where we used the short hand notation $\nu(f)=\int f \dd \nu$).
Since the $(\Xi_i)_i$ are IID under $\bbP$, $Z$ is a sum of IID Bernoulli variables with parameter smaller than $1/20$. By a standard large deviation computation there exists  a universal constant $c>0$
such that $\pi(A)\le e^{-c\alpha}.$

Under the measure $\mu_t$ the variables $(\Xi_i)^{\alpha}_{i=1}$ are still IID, since block independence is preserved through the iterations.
We have
\begin{equation}\label{40p_disc}
  \mu_{t}\left(\Xi_1\right) = p(p-1)2^{-t}+p=80(p-1)+p\ge 40p.
\end{equation}
To compute the second moment, we simply expand the product and compute the expectation of each type of term appearing in the expansion and the number of times they appear. This yields
\begin{multline}\label{secmondiscrete}
    \mu_{t}\left(  \Xi^2_1 \right)=
    p(p-1)(p-2)(p-3)     \mu_{t}\left(  \sigma_1\sigma_2 \sigma_3 \sigma_4\right)\\ + 6p (p-1) (p-2) \mu_{t}\left(  \sigma^2_1\sigma_2 \sigma_3 \right)+  4p(p-1)  \mu_{t}\left(  \sigma^3_1\sigma_2\right)\\
    + 3p(p-1) \mu_{t}\left(  \sigma^2_1\sigma^2_2\right)
    +p  \mu_{t}\left(  \sigma^4_1\right).
    \end{multline}
    and thus
    \begin{multline}
     \mu_{t}\left(  \Xi^2_1 \right)=
p(p-1)(p-2)(p-3)\left(3\times 4^{-t}(1-2^{-t})+ 8^{-t}\right)\\ +\left[6p (p-1)(p-2)+ 4p(p-1)\right] 2^{-t}+[3(p-1)p+p].
\end{multline}
Simplifying a couple of terms we obtain 
\begin{equation}
 \mu_{t}\left(  \Xi^2_1 \right)\le 3 p^2(p-1)^24^{-t}+6 p^2(p-1) 2^{-t}+ 3p^2 =3 \mu_t(\Xi_1)^2 
\end{equation}
Hence using Paley-Zygmund's Inequality we have %, on the event  $\{W_t\ge\frac{80(1 \vee \log (1/t))}{\alpha}\}$ (recall \eqref{40p_disc})
\begin{equation}
 \mu_{t}\left( \Xi_1 \ge  20p \right)\ge \mu_{t}\left( \Xi_1 \ge  \frac{1}{2}\mu_{t}(\Xi_1) \right)
\ge \frac{\mu_{t}(\Xi_1)^2 }{4  \mu_{t}(\Xi^2_1 )}=\frac{1}{12}.
\end{equation}
The variables $(X_i)_{i=1}^{\alpha}$ thus dominate IID Bernoulli with parameter $1/12$, therefore by a standard large deviation computation there exists a universal constant $c>0$ such that
$\mu_{t}\left( A^c \right)\le e^{-c\alpha}$,
which concludes the proof.
\end{proof}

\subsection{Remarks on more general initial distributions}\label{rem:extensions}
The proofs given above can be extended with minor modifications to prove the cutoff result \eqref{ubound1}-\eqref{lbound1} beyond the case of balanced initial distributions. Let us give the details of this result in the case where $\mu\in\cP(\O)$ has marginals $\mu(\si_i=+1)=p$ and $\mu(\si_i=-1)=1-p$, for some fixed $p\in(0,1)$.\\
In order to simplify the notation, it is convenient to take the spins $\sigma_i$ as elements of $\{-\sqrt{\frac{p}{1-p}},\sqrt{\frac{1-p}{p}}\}$ rather than $\{-1,1\}$ so that 
\begin{equation}\label{eq:margp}
\mu\(\si_i=\sqrt{\frac{1-p}{p}}\,\)=p\quad \text{and}\quad \mu\(\si_i=-\sqrt{\frac{p}{1-p}}\,\)=1-p.
\end{equation}
 Clearly, % (note that 
renaming the values of the spin has no effect on the results.
Note that $\sigma_i$ has variance $1$ under $\mu$. Note also that Lemma \ref{lem:rep} holds for any $\mu\in\cP(\O)$, and thus $\mu_t=\bbE[\mu_t^\xi]$, with  $\mu_t^\xi$ the measure with density 
\begin{equation}
	h^\xi_t(\si)=\frac{\dd\mu^{\xi}_t}{\dd\pi}(\si)=\prod_{i=1}^n \left(1+\sigma_i q^{\xi}(i)\right)\,,
\end{equation}
where $\pi$ denotes the product measure with marginals as in \eqref{eq:margp}, 
and $q^\xi$ is defined as before, see \eqref{eq:qxii}.\\
Then, the proof presented in the balanced case can be adapted to this setting and allows us to establish a cutoff phenomenon. In particular, the proof of Lemma \ref{lem:ubounds} applies verbatim and yields
$$ \| \mu_t-\pi \|_{\rm  TV} \le \bbE\left[\langle q^{\xi},q^{\xi}\rangle\right] = n \bbE\left[(q^\xi(1))^2\right] = n2^{-t}\,.$$
Similarly, one can follow the proof of Theorem \ref{monocase} provided one takes
\begin{align*}
	\alpha_n^{\xi}&:=\sqrt{np(1-p)}\log \left(\frac{1+\sqrt{\frac{1-p}{p}}\bar q^{\xi}}{1-\sqrt{\frac{p}{1-p}} \bar q^{\xi}} \right),\\
	\beta_n^{\xi}&:=n\bigg( p \log \left( 1+\sqrt{\frac{1-p}{p}}\bar q^{\xi} \right) + (1-p) \log \left( 1-\sqrt{\frac{p}{1-p}}\bar q^{\xi} \right) \bigg)\;.
\end{align*}
It is elementary to check that \eqref{distrybe} still holds in this context, and the subsequent arguments apply. The conclusion is that the same profile result in the statement \eqref{perfil} holds for all fixed $p\in(0,1)$ as $n\to\infty$. In particular, one has a lower bound on the total variation distance, and thus the cutoff displayed in \eqref{ubound1}-\eqref{lbound1} holds for all fixed $p\in(0,1)$. We note that while the upper bound applies without restrictions on $p$, some degree of nondegeneracy of $p$ as $n\to\infty $ is needed for the argument in the lower bound.

Moreover, the cutoff result holds also in the non-homogeneous case $\mu(\si_i=+1)=p_i$ and $\mu(\si_i=-1)=1-p_i$, as long as $p_i\in[\d,1-\d]$ for some fixed $\d\in(0,1/2)$. The proof of the upper bound goes exactly as above, with the only modification that the $p$ in \eqref{eq:margp} now depends on $i$. 
%in this case $g^\xi(i,\si_i) = \frac{\si_i(q^\xi(i) - u_i)}{1+\si_i u_i}$, where $u_i=2p_i-1$ may depend on $i$. 
%Repeating the previous argument one obtains again the upper bound \eqref{eq:qxib}.  
For the lower bound, more work is required since one has to introduce a non-homogeneous analogue of the monochromatic distribution. We omit the  details to maintain a more concise presentation.

Furthermore, we believe that a similar cutoff result also holds in the more general setting where the Boolean cube is replaced by an arbitrary product space $\bbX^n$, where $\bbX$ is a finite set, provided that the marginals $\mu_{\{i\}}$ are given by probability vectors with uniformly positive entries. However, we do not address this general problem %falls beyond the scope of 
in the present work. 

\section{Continuous time}

We start with a graphical construction for the solution $\mu_t=S_t(\mu)$ of the continuous time equation \eqref{cteq}. In analogy with the discussion in Section \ref{sec:graph_repr}, this involves a binary tree and a fragmentation process. The main difference is that the tree itself is given by a branching  random process.    
\subsection{Graphical construction and fragmentation in continuous time}\label{sec:graph_repr_ct}
 
 \subsubsection*{Tree considerations}
 We define the infinite rooted binary tree as $\mathbb T:=\{\emptyset\} \bigsqcup_{k\ge 1} \{0,1\}^k$ ($\bigsqcup$ being used to denote the disjoint union), and for $x\in \bbT$ we let $|x|$ denote the length of the sequence $x$.
 We equip $\mathbb T$ with the order $\prec$ by saying that $x \prec y$, if $|x| < |y|$ and $y$  can be written as $(x,z)$ for some $z\in \{0,1\}^{|y|-|x|}$ (with some slight abuse of notation $(x,z)$ denotes the concatenation of the two sequences). A finite rooted binary tree is a finite subset $\gamma$ of $\mathbb T$ which satisfies the following:
\begin{enumerate}
		\item For any $x\in \gamma$ and any $y\in \mathbb{T}$, if $y \prec x$ then $y\in \gamma$.
		\item For any $x\in \gamma$ either $\{(x,0),(x,1)\}\subset \gamma$ or $\{(x,0),(x,1)\}\cap \gamma=\emptyset$.
	\end{enumerate}
	The set of maximal elements in $\gamma$ (for $\prec$) that is, its leaves, is denoted by
	\begin{equation}\label{leaf}
	L(\gamma):= \{ x\in \gamma \ : \ \nexists y\in \gamma, \  x \prec y\}.
	\end{equation}
	Given a finite rooted binary tree $\gamma$ and $x\in \gamma$, we define $\gamma^x$, the subtree of $\gamma$  rooted at $x$, by 
	\begin{equation}\label{rooted}
	\gamma^x:= \{ y\in \bbT \ : \ (x,y)\in \gamma \}.
	\end{equation}
Thus, %%$\gamma^\emptyset=\gamma$, while 
if $\gamma\neq \{\emptyset\}$, then $\gamma^0,\gamma^1$ denote, respectively, the ``left'' and ``right'' subtrees of $\gamma$ after the first splitting.

\subsubsection*{Finite binary tree and measure recombination}
	For any finite binary rooted tree $\g$, and any $\mu\in\cP(\O)$, 
we define the distribution $p(\g,\mu)$ as follows.
Assign to each node $u$ of $\g$ a distribution $\nu_u\in\cP(\O)$ in such a way that each of the leaves is assigned the distribution $\mu$ and  the distribution at each internal node $u$ is computed as the collision product of the two distributions at the children of $u$. Then $p(\g,\mu)$ is the distribution at the root of the tree,  see Figure \ref{fig:fig2} for an example.

\begin{figure}[h]
\center
\begin{tikzpicture}[scale=0.75]
    \draw  (1,-3) -- (4,3);
    \draw  (3,-3) -- (2,-1);
    \draw  (4,-1) -- (3,1);
    \draw  (5,1) -- (4,3);
    
    \node[shape=circle, draw=black, fill = white, scale = 1.5]  at (3,1) {}; 
    \node[shape=circle, draw=black, fill = white, scale = 1.5]  at (2,-1) {};
    \node[shape=circle, draw=black, fill = gray, scale = 1.5]  at (1,-3) {};
    \node[shape=circle, draw=black, fill = gray, scale = 1.5]  at (3,-3) {};
    \node[shape=circle, draw=black, fill = gray, scale = 1.5]  at (4,-1) {};
    \node[shape=circle, draw=black, fill = white, scale = 1.5]  at (4,3) {};
    \node[shape=circle, draw=black, fill = gray, scale = 1.5]  at (5,1) {};
    \node at (1/4,-3) {$\mu$};
    \node at (9/4,-3) {$\mu$};
    \node at (1-.1,-1) {$\mu \circ \mu$};
    \node at (13/4,-1) {$\mu$};
    \node at (17/4,1) {$ \mu$};
    \node at (6/4-.2,1) {$(\mu\circ \mu) \circ \mu$};
    \node at (8/4-.3,3) {$((\mu\circ \mu) \circ \mu)\circ \mu$};
    \end{tikzpicture}
\caption{A possible tree $\g$ with 4 leaves, and the corresponding distribution at the root  $p(\g,\mu)=((\mu\circ \mu) \circ \mu)\circ \mu.$}
\label{fig:fig2}
\end{figure}
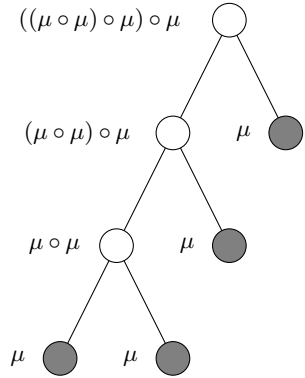

	\subsubsection*{The fragmentation process}
We consider a process $(\cT_t)_{t\ge 0}$ taking values in the set of finite rooted binary trees which we define as follows. We start with $\cT_0=\{\emptyset\}$. Then each leaf of the tree is equipped with a clock that rings after an exponentially distributed random time of mean one. When a clock rings on a leaf $x$ at time $t$ the vertices 
$(x,0)$ and $(x,1)$ are added to $\cT_t$ ($x$ gives birth to two new leaves).
Equivalently, the process $\cT_t$ can be seen as the result of first passage percolation on a complete binary rooted tree: We consider $(\cE_x)_{x\in \bbT}$ to be  a collection of exponential random variables of mean one. To each vertex we associate a time $\tau_x:= \sum_{y \prec x} \cE_y,$
and define 
\begin{equation}\label{defct}
\cT_t:= \{ x\in \mathbb T \ : \tau_x\le t    \}
\end{equation}
We  denote by $\bP$ the probability measure under which $\cT_t$ is defined, and write $\bE$ for the corresponding expectation.

\medskip

\begin{lemma}\label{lem:rep_ct}
Let $(\cT_t)$ be the branching process defined above. Then, for any $\mu\in\cP(\O)$, the solution to \eqref{cteq} is given by  
\begin{equation}\label{eq:solt}
 \mu_t= \bE\left[p(\cT_t,\mu)\right].
\end{equation}
\end{lemma}
\begin{proof}
We set $\rho_t:= \bE\left[p(\cT_t,\mu)\right]$ clearly we have $\rho_0=\mu$ so we only need to show that $\rho_t$ satisfies \eqref{cteq}. Setting  $\nu_t(\g):=\bP\left[\cT_t=\g\right]$ we have
\begin{equation}\label{eq:solt3}
\rho_t= \sum_{\g}\nu_t(\g)p(\g,\mu)=\nu_t(\{\emptyset\})\mu + \sum_{\g\neq \{\emptyset\}}\nu_t(\g)p(\g,\mu)\,.
\end{equation}
By definition of $\cT_t$ one has (recall the notation \eqref{rooted})
\begin{equation}\label{eq:solt2}
\nu_t(\{\emptyset\})=e^{-t}\quad  \text{ and } \quad \partial_t\nu_t(\g)= \ind_{\{\g\neq \{\emptyset\}\}}\nu_t(\g^{0})\nu_t(\g^{1})- \nu_t(\g)\,.
\end{equation}
The second identity in \eqref{eq:solt2} can be checked as follows. If $\gamma =\{\emptyset\}$, then $\nu_t(\g) = e^{-t}$ and the identity follows. If $\gamma \neq \{\emptyset\}$ then the first branching occurs at some time $s \in[0,t]$ and we have
$$ \nu_t(\g) = \int_0^t e^{-(t-s)} \nu_s(\g^{0}) \nu_s(\g^{1}) ds\;,$$
and differentiating this expression yields the asserted identity.
%The second identity in \eqref{eq:solt2} can be checked as follows. Let $P_t$ denote the Markov semigroup associated to the process $\{\cT_t\}$, so that $\nu_t(\g)=(P_t\ind_\g)(\{\emptyset\})$. Since a pair of leaves is added to the root with rate $1$, one has
%\[
%\partial_t\nu_t(\g)= (P_t\ind_\g)(u)-(P_t\ind_\g)(\{\emptyset\})\,,
%\]
%where $u$ denotes the binary tree obtained by adding two leaves to the root $\emptyset$. If $\g= \{\emptyset\}$ then clearly $(P_t\ind_\g)(u)=0$, while if $\g\neq \{\emptyset\}$ then, by independence of the two branches in $u$, $(P_t\ind_\g)(u)=\nu_t(\g^{0})\nu_t(\g^{1})$. This establishes  \eqref{eq:solt2}.
%  $\g_\ell,\g_r$ denote the left and right subtrees rooted at the children of the root of $\g$, respectively. 

By construction, if $\g\neq \{\emptyset\}$, one has that $p(\g,\mu)=p(\g^0,\mu)\circ p(\g^1,\mu)$.  Thus differentiating \eqref{eq:solt3} and using \eqref{eq:solt2} we obtain  
\begin{equation*}\begin{split}
\partial \rho_t &= - e^{-t}\mu + \sum_{\g\neq \{\emptyset\}}\left[\nu_t(\g^{0})\nu_t(\g^{1})p(\g^0,\mu)\circ p(\gamma^1,\mu)-\nu_t(\gamma)p(\gamma,\mu)\right]\\
&= - e^{-t}\mu + \sum_{\g',\gamma''}\nu_t(\g')\nu_t(\g'')p(\g',\mu)\circ p(\gamma'',\mu)-\sum_{\g\neq \{\emptyset\}}\nu_t(\gamma)p(\gamma,\mu) \\ 
& = \rho_t\circ \rho_t - \rho_t\, ,
\end{split}\end{equation*}
where the sum in $\gamma',\gamma''$ runs over all finite binary rooted trees. This concludes the proof.
\end{proof}
We turn to a pathwise description of the above construction. 
% Given a tree $\g$ as above, for any $x\in L(\g)$, write $|x|$ for the depth of $x$, that is the number of splittings in the path from $x$ to the root.
Recalling \eqref{leaf}, given a finite binary rooted tree $\gamma$ we let $U_i$, $i\in[n]$ be IID random variables taking values in $L(\g)$ with distribution given by
\begin{equation}\label{eq:uig}
P_\g(U_i = x) = 2^{-|x|}\,,\qquad x\in L(\g).
\end{equation}
Since $\sum_{x\in L(\g)}2^{-|x|} = 1$ for any $\g$, $P_\g$ is indeed a probability measure. Considering a random walk  that starts at the root of $\g$ and climbs the tree left or right with equal probability at each step, $P_{\gamma}$ is the law of the leaf at which this random walk ends. Equivalently, if we let $\cU$ be a uniform random variable in $[0,1]$ and let $x=x_{\gamma}(\cU)$ be the unique element of $L(\gamma)$ such that the first $|x|$ digits in the dyadic expansion of $\cU$ (which naturally encodes a random walk on $\bbT$) are given by $x$, then $P_{\gamma}$ is the law of $x_{\gamma}(\cU)$.

\medskip

\noindent Next, using $(U_i)_{i\in [n]}$, we define a random partition  $(A_x)_{x\in L(\gamma)}$ of $[n]$, by setting
 \begin{equation}\label{eq:uig2}
A_x:=\{ i\in[n]\,: \  U_i=x \}.
\end{equation}
We call again $P_\g$ the law of the random partition of $[n]$ obtained in this way, and write $E_\g$ for the expectation with respect to $P_\g$. 
\begin{lemma}\label{lem:rep_ct2}
For any finite binary rooted tree $\g$, for any $\mu\in\cP(\O)$,  
\begin{equation}\label{eq:uig3}
p(\g,\mu) = E_\g\left[ \bigotimes_{x\in L(\gamma)} \mu_{A_x}\right].
\end{equation}
\end{lemma}
\begin{proof}
We are going to prove \eqref{eq:uig3} by induction on the height $\max_{x\in L(\gamma)} |x|$ of the tree $\gamma$.
The case $\g=\{\emptyset\}$ is immediate. For any $\g\neq \{\emptyset\}$, recall that we have  $p(\g,\mu) = p(\g^0,\mu)  \circ
p(\g^1,\mu)$. Assuming, inductively, the validity of \eqref{eq:uig3} for $\g^0$ and $\g^1$, if $A$ is a uniformly random subset of $[n]$, one has
\[
p(\g,\mu) %= p(\g_\ell,\mu)  \circ
%p(\g_r,\mu) 
= E\left[ \left(\bigotimes_{x\in L(\gamma^0)} \mu_{A_x\cap A}\right)  \otimes   \left(\bigotimes_{y\in L(\gamma^1)} \mu_{A'_y\cap A}\right)\right],
\] 
where $(A_x)_{x\in L(\gamma^0)}$ and $(A'_y)_{y\in L(\gamma^1)}$ denote the partitions distributed according to $P_{\g^0}$ and $P_{\g^1}$ respectively, and the expectation $E$ is with respect to the independent triple $(A,(A_x),(A'_y))$. On the other hand, it is not hard to see that if $(A,(A_x),(A'_y))$ is as above, then
% are , then 
$$( (A'_x\cap A)_{x\in L(\gamma^0)} , (A'_y\cap A^{c})_{y\in L(\gamma^1)})  $$  has (after an appropriate relabelling which makes it a sequence indexed by $L(\gamma)$)  the distribution $P_\g$ defined by \eqref{eq:uig}-\eqref{eq:uig2}.
\end{proof}

Finally  we consider $\xi=(\xi(x),x\in\bbT)$  a field of IID random variables with law $\mu$ and let $\bbP$ denote the associated distribution.
Then for a finite binary rooted tree $\gamma$ we define %one has the density
\begin{equation}\label{eq:quenched}
%\frac{\dd \mu^{\xi,\cT_t}_t}{  \dd \pi}
\mu^{\xi}_\gamma(\si)=2^{-n}\prod_{i=1}^n (1+\si_i q^{\xi,\gamma}(i))\,,
\end{equation}
where 
\begin{equation}\label{def:qxi_cont}
q^{\xi,\gamma}(i):=\sum_{x\in L(\gamma)} 2^{-|x|}\xi_i(x).
\end{equation}
Repeating the argument of Lemma \ref{lem:rep}, we see that the probability $\mu^{\xi}_{\gamma}$ can be sampled by first sampling 
$(U_i)_{i\in [n]}$ IID with law $P_{\gamma}$ and then setting $\si_i=\xi_i(U_i)$.
Hence as a consequence of Lemma \ref{lem:rep_ct2} we have $
 p(\g,\mu) = \bbE\left[\mu^{\xi}_{\gamma}\right]$
and thus Lemma \ref{lem:rep_ct}  yields
\begin{equation}\label{eq:rep_ct21}
\mu_t =\bbE\otimes \bE \left[\mu^{\xi}_{\cT_t}\right].
\end{equation}

\subsection{The martingale $W_t$, size-biasing and spinal decomposition}\label{sec:spine}

We consider 
the random variable
\begin{equation}\label{defwt}
W_t=e^{t/2}\sum_{x\in L(\cT_t)} 4^{-|x|}\,,\qquad t\ge 0.
\end{equation}
and let $\cF_t:=\sigma\( \cT_s, s\le t \)$, $t\ge 0$ be the natural filtration associated to $\cT_t$, $t\ge 0$.

The paragraphs below will establish that this process is a uniformly integrable martingale. This is a known fact in the literature on branching processes, see~\cite{Uchiyama} and~\cite[Section 5]{Biggins}, which deal with a large class of branching processes and establish the uniform integrability of the associated additive martingale. However, in \cite{Uchiyama} a technical assumption excludes our process from the scope of the result while in~\cite[Section 5]{Biggins} the proof is provided in details for a discrete-time version of the branching process and the adaption to the continuous-time setting is left to the reader. Thus for the comfort of the reader, and although the result nor the  techniques involved are original, we provide a short and complete proof of uniform integrability (Lemma \ref{prop:UI}).

Recalling the definition \eqref{eq:uig} and the discussion and notation introduced below it, we have
\begin{equation}\label{decompo}
 W_t= E_{\cT_t}\left[ e^{t/2} 2^{-|U_1|}\right]=\int^1_0 e^{t/2} 2^{-|x_{\cT_t}(u)|}\dd u.
\end{equation}
We set $X_t(u):=x_{\cT_t}(u)$.
 Now the important observation is that for any fixed $u\in [0,1]$, the process $(|X_t(u)|)_{t\ge 0}$ is an intensity $1$ Poisson process. Indeed, from the construction presented in the last paragraph, given $u$ (which fixes an infinite path in $\bbT$ to be followed), the time spacings between the increments of $|X_t(u)|$ are IID exponentials.
Furthermore, the construction, and the memoryless property of exponential variables, implies that $|X_{t+s}(u)|-|X_t(u)|$ is independent of $\cF_t$.
From the above observation  we obtain that for every $u\in [0,1]$, $t,s\ge 0$, % we have 
\begin{equation}
\bE\left[ e^{(t+s)/2} 2^{-|X_{t+s}(u)|} \ | \ \cF_t\right]= e^{t/2} 2^{-|X_{t}(u)|}.
\end{equation}
Integrating over $[0,1]$ we obtain $\bE\left[ W_{t+s} \, | \, \cF_t\right]=W_t$ a.s.\ for all  $t,s\ge 0$. Summarizing, one has the following

\begin{proposition}\label{MartinW}
The process $(W_t)_{t\ge 0}$ is a martingale for the filtration $\cF_t$. 
In particular, 
\[
\bE\left[\sum_{x\in L(\cT_t)} 4^{-|x|}\right]=e^{-t/2}\,.
\]
\end{proposition}

\begin{remark}\label{Rk:UpperBd_ct}
	Using \eqref{eq:solt}, \eqref{eq:uig3} and the last proposition, we can deduce the upper bound \eqref{ctconv2}. Indeed, for any finite binary tree $\g$, let $F_\g$ be the event where full fragmentation has occurred, that is, where every $A_x$, $x\in L(\gamma)$,  is either a singleton or empty.  Then a union bound shows that 
	$$ \|p(\g,\mu) - \pi_\mu\|_{\rm TV} \le P_\g(F^c_\g) \le \frac12 n(n-1) \sum_{x\in L(\g)} 4^{-|x|}\;.$$
	Consequently,
	\begin{align*}
		\|S_t(\mu)-\pi_\mu\|_{\rm TV} &\leq \bE\left[ \|p(\cT_t,\mu) - \pi_\mu\|_{\rm TV} \right]\\
		&\leq \bE[P_{\cT_t}(F^c_{\cT_t})] \leq \tfrac12\,n(n-1)e^{-t/2}.
	\end{align*}
\end{remark}

\bigskip

%\noindent 
Being a nonnegative martingale, $W_t$ converges a.s.\  to a limit $W_{\infty}$. 
We are going to prove now that $W_t$ is uniformly integrable, so that $\bbE[W_{\infty}]=1$ and hence the limit is non-trivial.
For this we rely on the study of the size biased measure $\wt \bP_t$ defined by
$\dd \wt \bP_t= W_t \dd  \bP$.
From \eqref{decompo} we have 
\begin{equation}
\bP_t= \int^1_0 \wt \bP_{t,u} \dd u \quad \text{ where } \quad \dd \wt \bP_{t,u}=e^{t/2} 2^{-|X_t(u)|}\dd \bP.
\end{equation}
Informally, the change of measure $\bP\to \wt \bP_{t,u}$ has the effect of slowing the exponential clocks along the path in $\bbT$ which starts from the root and follows the dyadic expansion of $u$. 
We denote this path by $(u_i)_{i\ge 0}$ and refer to it as the \textit{spine}. Note that the change of  measure $\bP\to \wt \bP_{t,u}$ has no effect on the distribution of the clocks for vertices outside the spine $(\cE_x)_{x\in \bbT\setminus \{u_i\}_{i\ge 0}}$, which remain independent exponential variables of parameters $1$ that are indendent of the process $(|X_s(u)|)_{s\ge 0}$: by construction $|X_t(u)|_{t\ge 0}$ is a function of $(\cE_{u_i})_{i\ge 0}$ and is therefore independent of $(\cE_x)_{x\in \bbT\setminus \{u_i\}_{i\ge 0}}$.
It is then easy to describe the distribution of $(|X_s(u)|)_{s\ge 0}$ under $\wt \bP_{t,u}$. It is an inhomogeneous Poisson process which has intensity $1/2$ on $[0,t]$ and $1$ on $(t,\infty)$, cf. Appendix \ref{app:Poisson} for a proof of this fact. 
Using this description we prove the following

\begin{lemma}\label{prop:UI}
 The martingale $(W_t)_{t\ge 0}$ is uniformly integrable.

\end{lemma}
% As we said above, this result is already covered by a more general result of Biggins~\cite{Biggins}.

 \begin{proof}
 	We are going to show that $\bbE[W^{1+\theta}_t]$ is uniformly bounded in $t$ for some $\theta\in(0,1)$.
 	We have 
 	\begin{equation}
 		\bbE\left[ W^{1+\theta}_t\right]=\int^1_0 \wt\bE_{t,u}\left[W^{\theta}_t\right] \dd u,
 	\end{equation}
 	where $\wt\bE_{t,u}$ denotes the expectation with respect to $\wt\bP_{t,u}$.
 	Recalling \eqref{defct}, we set $T_i:=\tau_{u_i}$. For $i\ge 1$ we let  $v_{i}\in \mathbb{T}$ be the  child of $u_{i-1}$ which does not belong to the spine. By definition for any $t\ge 0$, we have  $t\in [T_{|X_t(u)|}, T_{|X_t(u)|+1})$. We then rewrite \eqref{defwt} as follows
 	\begin{equation}
 		W_t= e^{t/2}\left( 4^{-|X_t(u)|}+ \sum_{i=0}^{|X_t(u)|-1} \sum_{y\in L(\cT^{v_{i+1}}_t)} 4^{-i-1-|y|}\right).
 	\end{equation}
 	so that using subadditivity 
 	\begin{equation}
 		W^{\theta}_t \le e^{\theta t/2}\left( 4^{-\theta|X_t(u)|}+ \sum_{i=0}^{|X_t(u)|-1} \left(\sum_{y\in L(\cT^{v_{i+1}}_t)} 4^{-i-1-|y|}\right)^{\theta}\, \right).
 	\end{equation}
 	By construction, under $\wt \bP_{t,u}$ conditionally given $(X_s(u))_{s\in[0,t]}$, for $i\le |X_t(u)|-1$ the subtree $(\cT^{v_{i+1}}_t)$ is distributed like $\cT_{t-T_{i+1}}$ under $\bP$. Hence,
 	using Jensen's inequality and Proposition \ref{MartinW}, we have
 	\begin{multline}
 		\wt\bE_{t,u}\left[ \left(\sum_{y\in \cT^{v_{i+1}}_t} 4^{-|y|}\right)^{\theta} \, \Big|\, (X_s(u))_{s\in[0,t]} \right]\\ \le \wt\bE_{t,u}\left[ \sum_{y\in \cT^{v_{i+1}}_t} 4^{-|y|} \ | \ (X_s(u))_{s\in[0,t]} \right]^{\theta} = e^{\theta(T_{i+1}-t)/2}. 
 	\end{multline}
 	Therefore, 
 	\begin{equation}
 		\wt\bE_{t,u}\left[ W^{\theta}_t\right]\le   \wt\bE_{t,u}\left[ e^{\theta t/2}4^{-\theta|X_t(u)|}+ \sum_{i=0}^{|X_t(u)|-1} 4^{(-i-1)\theta}e^{\frac{\theta T_{i+1}}{2}}\right].
 	\end{equation}
 	By the observation in Appendix \ref{app:Poisson}, one has
 	$\wt\bE_{t,u}\left[ e^{\theta t/2}4^{-\theta|X_t(u)|}\right]=e^{\frac{t}{2}(\theta+4^{-\theta} -1)}$, which is bounded uniformly in $t$ if $\theta\in [0,1/2]$.
 	On the other hand we have 
 	\begin{equation}\label{sumz}
 		\wt\bE_{t,u}\left[ \sum_{i=0}^{|X_t(u)|-1} 4^{(-i-1)\theta}e^{\theta \frac{T_{i+1}}{2}}\right]\le \sum_{i=0}^{\infty} 4^{(-i-1)\theta} \wt\bE_{\infty}\left[ e^{\frac{\theta T_{i+1}}{2}}\right],
 	\end{equation}
 	where under $\wt\bE_{\infty}$, the variables $T_i$ are sums of IID exponential with mean $2$. Since 
 	$\wt\bE_{\infty}\left[  e^{\frac{\theta T_{i+1}}{2}}\right]= (1-\theta)^{-(i+1)}$, the sum in \eqref{sumz} is bounded provided $4^{\theta}(1-\theta)>1$.
 	This is satisfied for instance for $\theta=1/4$.
  \end{proof}

The following result provides some information about the distribution of $W_{t}$ and is proved in Appendix \ref{sec:important}. 

%\cyril{Je ne suis pas sûr que la preuve fonctionne quand $\gep$ est ``grand''. Ca semble problematique pour l'upper bound.}
\begin{lemma}\label{important}
There exist constants $c,C>0$ such that for every $t\in (0,\infty]$ and $\gep\in(0,1)$
\begin{equation}
 \exp\left(- \frac{C(1\vee \log(1/ t))}{\gep}\right) \le \bP\left( W_t \le \gep\right) \le 2\exp\left(- \frac{c(1\vee \log(1/ t))}{\gep}\right).
\end{equation}
In particular, $\bP(W_{\infty}>0)=1.$ 
\end{lemma}

\subsection{Upper bound for continuous time}

Before we proceed with the upper bound, let us state and prove an  identity which is important  for what follows:
\begin{equation}\label{Eq:qxi_cont}
	\bE\otimes\bbE\left[\sum_{i=1}^n q^{\xi,\cT_t}(i)^2\right] = n\,e^{-t/2}.
\end{equation}
Indeed,  the $q^{\xi,\cT_t}(i)$, $i\in [n]$, as defined in \eqref{def:qxi_cont}, have the same laws, and therefore 
\begin{equation}\label{Eq:qxi_cont2}
	\bbE\left[\sum_{i=1}^n q^{\xi,\cT_t}(i)^2\right] = n \bbE\left[q^{\xi,\cT_t}(1)^2\right] = n\sum_{x\in L(\cT_t)} 4^{-|x|} =ne^{-t/2}W_t.
\end{equation}
 To compute the above expectation we used that $\xi_1(x)$, $x\in L(\cT_t)$ are IID $\{\pm 1\}$-valued r.v.~with mean $0$.
 Using Proposition \ref{MartinW}, we obtain the desired result \eqref{Eq:qxi_cont}.
% \begin{equation}
% \bE\otimes\bbE\left[\sum_{i=1}^n q^{\xi,\cT_t}(i)^2\right] = n \bE\left[\sum_{x\in L(\cT_t)} 4^{-|x|} \right]= n\,e^{-t/2}.
% \end{equation}
\begin{lemma}\label{lem:ubounds_cont}
	For every integer $n\ge 1$, $t\ge 0$ and $\mu\in\cB_n$, setting $r:= n e^{-t/2}$, 
	\begin{equation}\label{eq:rootbd_cont}
		\|\mu_t-\pi\|_{\rm TV}\le r\,.
	\end{equation}
	Moreover, there exists a constant $C>0$ such that 
	\begin{equation}\label{unilow_cont}
		\|\mu_t-\pi\|_{\rm TV}\le 1-\frac{1}{2}e^{-C\sqrt{r(1\vee \log(1/ t))}}.
	\end{equation}
\end{lemma}
\begin{remark}\label{thevalueoft}
	
	We always have $\|\mu_t-\pi\|_{\rm TV}\le 1-2^{-n}$. Indeed, $\pi(\sigma) = 2^{-n}$ for all $\sigma \in \Omega$ while there exists $\sigma'\in \Omega$ such that $\mu_t(\sigma') \ge 2^{-n}$ and therefore
	$$\|\mu_t-\pi\|_{\rm TV} = \sum_{\sigma} (\pi(\sigma) - \mu_t(\sigma))_+ \le \sum_{\sigma \ne \sigma'} \pi(\sigma) = 1-2^{-n}\;.$$
	Consequently \eqref{unilow_cont} is not optimal for very small values of $t$ (when $\log (1/t)$ is of order $n$). On the other hand this bound is in a sense optimal for $t\ge e^{-1/n}$, see Proposition \ref{lbcont}.
\end{remark}
	\begin{proof}
	Repeating the exact same steps as in the proof of Lemma \ref{lem:ubounds}, we obtain
	$$  \| \mu_t-\pi \|_{\rm  TV} \le \bE\otimes\bbE\left[\langle q^{\xi,\cT_t},q^{\xi,\cT_t}\rangle\right]\;.$$
 The statement \eqref{eq:rootbd_cont} then follows 	by \eqref{Eq:qxi_cont}.

	To prove \eqref{unilow_cont}, we first observe that from \eqref{eq:rootbd_cont}, we only need to worry about the case where $r\ge 1/2$.
	Secondly if $\log(1/t) >n$ then $\sqrt{r(1\vee \log(1/ t))} \ge n e^{-t/4}$ so that Remark \ref{thevalueoft} already provides a bound of the right order. Consequently we only need to prove the result when  $\log(1/t)\le n$.
	 We proceed by combining the ideas used in the proof of Lemma \ref{lem:ubounds2} with Lemma \ref{important}. We have
\begin{align*}
&	 \| \mu_t-\pi \|_{\rm  TV} \le \bP\otimes\bbP\left[  \phi\left( \sqrt{ e^{\langle q^{\xi,\cT_t}, q^{\xi,\cT_t}\rangle}-1} \right)\right]\\
	 &\le  \bP\otimes\bbP\(\langle q^{\xi,\cT_t}, q^{\xi,\cT_t}\rangle \ge 2r \gep\) + \bE\otimes\bbE\left[\mathbf{1}_{\{\langle q^{\xi,\cT_t}, q^{\xi,\cT_t}\rangle < 2r \gep\}}\phi\left( \sqrt{ e^{\langle q^{\xi,\cT_t}, q^{\xi,\cT_t}\rangle}-1} \right)\right]\\
	 &\le 1- \bP\otimes\bbP\(\langle q^{\xi,\cT_t}, q^{\xi,\cT_t}\rangle < 2r \gep\) + \phi\(\sqrt{e^{2r\gep} - 1}\)  \bP\otimes\bbP\(\langle q^{\xi,\cT_t}, q^{\xi,\cT_t}\rangle < 2r \gep\)\\
	 &\le 1 - \(1- \phi\(\sqrt{e^{2r\gep} - 1}\)\)\bP\otimes\bbP\(\langle q^{\xi,\cT_t}, q^{\xi,\cT_t}\rangle < 2r \gep\)\;.
\end{align*}
From \eqref{Eq:qxi_cont2} we have
$\bbE\left[ \langle q^{\xi,\cT_t}, q^{\xi,\cT_t}\rangle\right]=r W_t$, so that by Markov's Inequality
$$\bbP\left( \langle q^{\xi,\cT_t}, q^{\xi,\cT_t}\rangle < 2r W_t  \right)= 1- \bbP\left( \langle q^{\xi,\cT_t}, q^{\xi,\cT_t}\rangle\ge 2r W_t\right) \ge  1/2.$$ Therefore, for any $\gep \in(0,1)$,
\begin{equation*}
	\bP\otimes \bbP\(\langle q^{\xi,\cT_t}, q^{\xi,\cT_t}\rangle < 2r \gep\) \ge \bP\otimes \bbP\( W_t \le \gep  \ ; \ \langle q^{\xi,\cT_t}, q^{\xi,\cT_t}\rangle < 2r W_t\)\ge \frac12\bP\(W_t\le \gep\).
\end{equation*}
Hence for any $\gep>0$ we have
\begin{equation}
\| \mu_t-\pi \|_{\rm  TV}\le  1-  \frac{1}{2}\bP\(W_t\le \gep\)\left(1- \phi\left(\sqrt{ e^{2\gep r}-1}\right)\right).
\end{equation}
We chose $\gep=\sqrt{\frac{1\vee \log (1/t)}{3r}}$. 
Note that $\gep<1$ if $t\ge 1$ since we assumed $r\ge 1/2$. If instead $t< 1$, then  $\gep\le \sqrt{\frac{e^{1/2}\log (1/t)}{3n}}<1$ because of our assumption $\log (1/t)\le n$. We can thus apply Lemma \ref{important} and recalling that $\phi(x) = x^2/(1+x^2)$ for $x\ge 1$, we obtain the desired result.
 
%  \cyril{Je crois qu'il y a un petit probleme quand $\gep \ge 1$ (voire même quand $\gep > 1/4$)}
\end{proof}

\subsection{Profile for monochromatic initial states}

We now prove Theorem \ref{monocasect}. We take the monochromatic distribution $ \mu =  \tfrac12\,\ind_{-} + \tfrac12\,\ind_{+}$ as initial state and rely on the construction of Section \ref{sec:graph_repr_ct}. More precisely given a process  $(\cT_t,t\ge 0)$ and IID r.v. $(\xi(x),x\in \bbT)$ with law $\mu$ we use the representation \eqref{eq:rep_ct21} for $\mu_t$. Note that for this choice of $\mu$, the value of ${q}^{\xi,\cT_t}(i)$ from \eqref{def:qxi_cont} does not depend on $i$. Letting $\bar{q}^{\xi,\cT_t}$ denote their common value, we have
\begin{equation} \label{whatrhotis}
	\rho_t(\si):=\frac{\dd \mu_t}{\dd \pi}(\si)=\bE\otimes \bbE\left[ \prod_{i=1}^n \(1+\si_i \bar q^{\xi,\cT_t}\)\right].
\end{equation}
In analogy with  the discrete time case, we have the following convergence in distribution. Recall that $W_{\infty}$ denotes the limit of the martingale $W_t$ associated with the process $\cT_t$ (cf. Proposition \ref{MartinW}).

\begin{proposition}\label{maincase_cont}
	Given $r>0$, if $t_n:= 2\log(n/r)$, then  for any bounded continuous function 
	$F:\bbR_+\to\bbR$,
	\begin{equation} \label{inlaw_cont}
		\lim_{n\to \infty}  \int F(\rho_{t_n}(\si))\pi(\dd \si)= \int \frac{e^{-\frac{z^2}{2}}}{\sqrt{2\pi}}\,F\left( \bE\left[\gamma_{rW_\infty}(z)\right] \right)  \dd z\;,
	\end{equation}
	where $\gamma_{rW_\infty}$, denotes the function defined in \eqref{eq:gamma_s} with $s=rW_\infty$.
%
%	with
%	\[
%	\gamma_u(z):=\frac{e^{\frac{uz^2}{2(u+1)}}}{\sqrt{u+1}}\;,\quad z\in\R\;.
%	\]
\end{proposition}

%Theorem \ref{monocasect} then follows from this result and the very same arguments as those presented in the proof of Theorem \ref{monocase}: in a nutshell, we apply the last proposition with $F(x) = 1+x - |x-1|$, and we use the fact that the contribution coming from $1+x$ in the corresponding integral is equal to $2$ for any $n$ and in the limit.\\

Before proving Proposition \ref{maincase_cont} we show how to use it to prove our profile result in continuous time. 
\begin{proof}[Proof of Theorem \ref{monocasect}]
	We are going to prove that
	\begin{equation}\begin{split}\label{conseq_cont}
			\lim_{n\to \infty}
			\int \left|\rho_{t_n}(\si)-1\right|\pi(\dd \si)
			=\int \frac{e^{-\frac{z^2}{2}}}{\sqrt{2\pi}}\left|\bE\left[\gamma_{rW_\infty}(z)\right]-1\right| \dd z.
	\end{split}\end{equation}
Note that this establishes \eqref{profilect} with the representation \eqref{eq:nulambda}, where $W=W_\infty$. 	

Since the map $F:x\mapsto 1+x - |x-1|$ is bounded and continuous on $\R_+$, Proposition \ref{maincase_cont} yields
	\begin{align}
		&\lim_{n\to \infty}\int \left[(1+\rho_{t_n}(\si))-|\rho_{t_n}(\si)-1|\right]\pi(\dd \si) \\
		&\qquad \qquad = \int \frac{e^{-\frac{z^2}{2}} }{\sqrt{2\pi}}\left[ \left(
		1+\bE\Big[\gamma_{rW_\infty}(z)\Big]\right) - \left|\bE\left[\gamma_{rW_\infty}(z)\right]-1\right|\right]\dd z.
	\end{align}
	Since $\int (1+\rho_{t_n}(\si)) \pi(\dd \si)=2= \int\frac{e^{-\frac{z^2}{2}} }{\sqrt{2\pi}}\left(1+\bE\left[\gamma_{rW_\infty}(z)\right]\right)\dd z$, \eqref{conseq_cont}  follows.
\end{proof}

The proof of Proposition \ref{maincase_cont} follows the same line of argument as its discrete counterpart, Proposition \ref{maincase}, however some steps need some additional work: for the sake of clarity, we thus provide a rather complete proof.

\begin{proof}[Proof of Proposition \ref{maincase_cont}]
	The convergence in the statement can be interpreted as convergence in distribution of the r.v.~$\rho_{t_n}$ under $\pi$ towards the r.v.~$\bE\gamma_{rW_\infty}(Z)$ with $Z\sim\cN(0,1)$. %\bP\otimes\cN(0,1)$. 
	Without loss of generality, we can assume that $F:\R_+\to\R$ is Lipschitz and bounded. We then set
	$$ G(a) := \bE[\gamma_{rW_\infty}(a)]\;,\quad a\in\R\;.$$
	It is not hard %straightforward 
	to check that $G$ is continuous. Assume that
	\begin{equation}\label{need2_cont}
		\lim_{n\to \infty} \Big\vert \int F(\rho_{t_n}(\si)) \pi(\dd \si) - \int F\circ G(\bar \si_n)\pi(\dd \si) \Big\vert=0\;,
	\end{equation}
	where $\bar \si_n: = \frac1{\sqrt n}\sum_{i=1}^n \si_i$. With this convergence at hand, it only remains to show that
	\begin{equation} \label{inlawbar_cont}
		\lim_{n\to \infty}  \int F\circ G(\bar \si_n) \pi(\dd \si)= \int \frac{e^{-\frac{z^2}{2}}}{\sqrt{2\pi}} F\circ G(z)  \dd z\;.
	\end{equation}
	Since $F\circ G$ is bounded and continuous, this convergence follows simply from the fact that $\bar \si_n$ converges in distribution, under $\pi$, to a standard Gaussian.\\
	
	We are left with the proof of \eqref{need2_cont}. As in the discrete case, we observe that 
	\[ h^{\xi,\cT_t}(\sigma):=	\prod_{i=1}^n (1+\si_i \bar q^{\xi,\cT_t})
= e^{\alpha_n^{\xi,\cT_{t}} \bar \si_n+ \beta_n^{\xi,\cT_{t}}},
	\]
	where
	\begin{equation}\label{degalfbet_cont}
		\alpha_n^{\xi,\cT_{t}}:=\frac{\sqrt{n}}{2}\log \left(\frac{1+\bar q^{\xi,\cT_{t}}}{1-\bar q^{\xi,\cT_{t}}} \right), \quad \beta_n^{\xi,\cT_{t}}:=\frac{n}{2}\log \left( 1-(\bar q^{\xi,\cT_{t}})^2 \right) ,
	\end{equation}
	with an appropriate convention for the special case $\bar q^{\xi,\cT_{t_n}}=1$.\\

	We need the following convergence in distribution, which is slightly more delicate than its discrete counterpart since one has %we need 
	to take into account the randomness coming from the tree.
	\begin{lemma}\label{lem:alphabeta_cont}
		Under $\bbP\otimes \bP$, the pair $(\alpha_n^{\xi,\cT_{t_n}}, \beta^{\xi,\cT_{t_n}}_n)$ converges in distribution towards the pair $(Y_{r}, -\frac12 Y_{r}^2)$, where $Y_{r}= \sqrt{W_\infty}Z_r$, and  $Z_r\sim\cN(0,r)$ is independent of $W_{\infty}$.
	\end{lemma}
	\begin{proof}[Proof of Lemma \ref{lem:alphabeta_cont}]
	Using Taylor's expansion in the expressions \eqref{degalfbet_cont}, and the fact that $\sqrt{n}e^{-t_n/4}=\sqrt{r}$ the result boils down to proving 
	\begin{equation}\label{toy1}
	\lim_{t\to \infty}e^{t/4} \bar q^{\xi,\cT_{t}}\stackrel{(d)}{=} Y_1.
	\end{equation}
	Recall that
	$ \bar q^{\xi,\cT_{t}} = \sum_{x\in L(\cT_t)} 2^{-|x|}\xi_1(x).$
	In order to prove \eqref{toy1} we compute the corresponding Fourier transform. Averaging first w.r.t.\ to the IID variables $\xi(x)$ we obtain that for all $b\in\R$
	$$ \bbE\left[ e^{i b e^{t/4} \bar q^{\xi,\cT_{t}}} \right] = \exp\Big(\textstyle{\sum_{x\in L(\cT_{t})}} \log \cos \(b e^{t/4} 2^{-|x|}\)\Big)\;.$$
	Now we claim that we have the following convergence in probability
	\begin{equation}\label{themax}
	 \lim_{t\to\infty}\max_{x\in L(\cT_t)} 2^{-|x|}e^{t/4}=0. 
	\end{equation}
	This implies via the use of Taylor expansion the following convergence in probability
	\begin{equation}
	\lim_{t\to \infty} \sum_{x\in L(\cT_{t})} \log \cos (b e^{t/4} 2^{-|x|})= -\lim_{t\to \infty}\sum_{x\in L(\cT_{t})} \frac12 b^2 e^{t/2} 4^{-|x|} = -\frac12 b^2 W_{\infty}.
\end{equation}
The Dominated Convergence Theorem  then yields that for all $b\in\R$
	$$ \lim_{t\to\infty} \bE\otimes\bbE\left[ e^{i b e^{t/4} q^{\xi,\cT_{t}}} \right] = \bE\left[\exp\(-\tfrac12 b^2 W_\infty\)\right]\;.$$
	To conclude the proof we just need to justify \eqref{themax}.
Repeating the proof of Proposition \ref{MartinW}, if $X$ is a Poisson r.v.~of parameter $t$ then for any given $\lambda \ge 2$ one has
$$\bE\left[ \sum_{x\in L(\cT_t)} \lambda^{-|x|} e^{t(1-\frac{2}{\lambda})}\right]=\bE\left[ \sum_{x\in L(\cT_t)}2^{-|x|} (\lambda/2)^{-|x|} e^{t(1-\frac{2}{\lambda})}\right] = \E[(\lambda/2)^{-X} e^{t(1-\frac{2}{\lambda})}] = 1.$$
Thus, by Markov's inequality we have
\begin{equation}
	\bP\left[ \exists x\in L(\cT_t), |x|\le m\right]\le \lambda^m e^{-t(1-\frac{2}{\lambda})}. 
\end{equation}
This implies \eqref{themax} provided
$$ \frac14\frac{\log \lambda}{\log 2} < 1 - \frac2{\lambda}\;,$$
which holds true with $\lambda = 2^{5/2}$ for instance.
\end{proof}
\noindent Given  $a\in \bbR$, we set  
	$G_n(a) :=  \bE\otimes\bbE\left[e^{\alpha_n^{\xi,\cT_{t_n}} a+ \beta_n^{\xi,\cT_{t_n}}}\right]$
	so that
	\begin{equation}\label{distrybe1_cont}
		\rho_{t_n}(\si)=G_n(\bar \si_n).
	\end{equation}
	The next lemma follows from the same arguments as in the proof of Lemma \ref{unifconv}.
	\begin{lemma}\label{lem:gG_cont}
	For all $a\in\R$,
	\begin{equation}
		\lim_{n\to \infty} G_n(a) = G(a)\;,
	\end{equation}
	and the convergence is uniform on the interval $[-A,A]$ for any $A>0$.
	\end{lemma}
\noindent
	We can now conclude the proof of \eqref{need2_cont}. We show that for any given $A>0$,
	\begin{equation}\label{Eq:monotoprove1_cont}
		\limsup_{n\to\infty} \int  \Big\vert F(\rho_{t_n}(\si)) - F\circ G(\bar \si_n) \Big\vert \mathbf{1}_{\{\vert \bar \si_n\vert \le A\}}\pi(\dd \si) = 0\;,
	\end{equation}
	and that
	\begin{equation}\label{Eq:monotoprove2_cont}
		\lim_{A\to\infty} \limsup_{n\to\infty} \int  \Big\vert F(\rho_{t_n}(\si)) -  F\circ G(\bar \si_n) \Big\vert \mathbf{1}_{\{\vert \bar \si_n\vert > A\}}\pi(\dd \si) = 0\;.
	\end{equation}
	The latter follows as in \eqref{Eq:monotoprove2}. 
%	By Hoeffding's inequality, for all $n\ge 1$
%	$$ \pi(\{\vert \bar \si_n\vert > A\}) \le 2 e^{-A^2/2}\;,\quad A>0\;.$$
%	Since $F$ is bounded, we easily deduce \eqref{Eq:monotoprove2_cont}.\\
	To prove \eqref{Eq:monotoprove1_cont}, since $F$ is Lipschitz, it suffices to show that for any given $A>0$
	\begin{equation}\label{Eq:monotoprove1bis_cont}
	\limsup_{n\to\infty} \int  \Big\vert \rho_{t_n}(\si) - G(\bar \si_n) \Big\vert \mathbf{1}_{\{\vert \bar \si_n\vert \le A\}}\pi(\dd \si) = 0\;.
	\end{equation}	
	Since $\rho_{t_n}(\si) = G_n(\bar \si_n)$, Lemma \ref{lem:gG_cont} allows us to conclude.
\end{proof}

\begin{remark}
	If we set for all $a\in\R$,
	$ G_n(a,\cT_{t}) :=  \bbE\left[e^{\alpha_n^{\xi,\cT_{t}} a+ \beta_n^{\xi,\cT_{t}}} \right]$
	and we define the random probability measure on $\O$
	$ \nu_{n}(\si) := G_n(\bar \si_n,\cT_{t_n}) \pi(\si)\;,$
	then almost surely
	$$ 			\lim_{n\to \infty} \| \nu_{n} - \pi\|_\tv = \varphi\big(rW_\infty\big)\;.$$
	Note that $ \mu_{t_n} = \mathbf{E}[\nu_{n}]\;,$
	so that
	$  \| \mu_{t_n} - \pi\|_\tv \le \mathbf{E}\big[ \| \nu_{n} - \pi\|_\tv \big]\;.$
	Passing to the limit we thus get
	$$\frac12 \int \frac{e^{-\frac{z^2}{2}}}{\sqrt{2\pi}}\left|\mathbf{E}\left[\gamma_{rW_\infty}(z)\right]-1\right| \dd z \le \mathbf{E}\left[\varphi\big(rW_\infty\big)\right]\;.$$
	When $r=e^{-\lambda/2}$, the left hand side above is precisely the function $f(\lambda)$ in \eqref{eq:nulambda}.
%	This is consistent since
%	\begin{align*}
%		\int \frac{e^{-\frac{z^2}{2}}}{\sqrt{2\pi}}\left|\mathbf{E}\Big[\gamma_{rW_\infty}(z)\Big]-1\right| \dd z &= \int \frac{e^{-\frac{z^2}{2}}}{\sqrt{2\pi}}\left|\mathbf{E}\Big[\gamma_{rW_\infty}(z)-1\Big]\right| \dd z\\
%		&\le\int \frac{e^{-\frac{z^2}{2}}}{\sqrt{2\pi}}\mathbf{E}\Big[\left|\gamma_{rW_\infty}(z)-1\right|\Big] \dd z\\
%		&=\mathbf{E}\Big[ \int \frac{e^{-\frac{z^2}{2}}}{\sqrt{2\pi}}\left|\gamma_{rW_\infty}(z)-1\right| \dd z\Big]\\
%		&= \mathbf{E}\Big[ \| \cN(0,1+rW_\infty) - \cN(0,1)\|_\tv\Big]\\
%		&= \mathbf{E}[\varphi\big(rW_\infty\big)]\;.
%	\end{align*}
%	
\end{remark}

\subsection{Lower bound for continuous time}

\begin{proposition}\label{lbcont}
  There exists $c>0$ and $n_0$ such that for every $n\ge n_0$ and $t\ge e^{-\frac{n}{10^4}}$, setting $r= n e^{-t/2}$, one has
 \begin{equation}
  d_n(t)\ge 1- 2e^{-c\sqrt{r(1\vee \log (1/t))}}\,.
 \end{equation}
\end{proposition}
For a comment on  the restriction for $t$ we refer the reader to Remark \ref{thevalueoft}. Note also that Lemma \ref{lem:ubounds_cont} and Proposition \ref{lbcont} complete the proof of Theorem \ref{th:continuous} (the case $t\le e^{-\frac{n}{10^4}}$ in the theorem follows by monotonicity of $t\mapsto d_n(t)$).
% \noindent Neither the proposition nor its proof cover the case $t\in(0,1]$ which is technically demanding and presents a limited interest.

\begin{proof}
The proof follows the same plan as that of Proposition \ref{prop:lowerbd}.
We split the set of coordinates into groups of size  $p:=\lfloor \sqrt{80 n} e^{t/4} (1\vee \log(1/t))^{-1/2} \rfloor$, with a leftover and let 
$\alpha:=\lfloor n/p\rfloor$ denote the number of blocks thus obtained. 
The result is equivalent to showing that (for a different $c>0$) 
 $d_n(t)\ge 1- 2e^{-c\alpha}$.
 If $\alpha < 1$ this trivially holds provided $c$ is small enough. We thus assume that $\alpha \ge 1$.
In the remainder of the computation we neglect the effect of integer rounding for better readability.
We are going to find a measure $\mu$ and an event $A$ such that
\begin{equation}
 \pi(A)\le e^{-c\alpha} \;\;\text{ and }\;\; \mu_t(A^{c})\le 3e^{-c\alpha}.
\end{equation}
Adjusting the value of the constant $c$ will then conclude the proof.
% Let us assume for simplicity that $\sqrt{r}$ and $n/\sqrt{r}=p$ is are integers.
% We split our set of coordinates into $\alpha$ groups of cardinality $p=\lfloor n/\alpha \rfloor$ and a remainder.

We let $\Xi_i$ denote the squared magnetization of the $i$-th group.
\begin{equation}
 \Xi_i:= \left(\sum_{j=(i-1)p+1}^{ip} \sigma_j\right)^2.
\end{equation}
We set
\begin{equation}
X_i:=\ind_{\{\Xi_i\ge 20p\}}, \quad  Z:=\sum_{i=1}^{\alpha}  X_i \quad \text{ and }
 A:=\left\{ Z \ge \frac{\alpha}{15} \right\}.
\end{equation}
Finally we let the initial condition be monochromatic in each block and independent between blocks, that is  $\mu= (\nu_p)^{\otimes \alpha}\otimes \nu_{n-p\alpha}$. As shown in the proof of Proposition \ref{prop:lowerbd} we have
$\pi(A)\le e^{-c\alpha}.$

Now, under the measure $\mu_t$ the variables $(\Xi_i)^{\alpha}_{i=1}$ are not IID, for this reason, recalling \eqref{eq:rep_ct21}, we rather consider $\mu_{\cT_t}$ defined by 
$ \mu_{\cT_t}:= \bbE\left[ \mu^{\xi}_{\cT_t}\right]$.
It follows that (neglecting the effect of integer rounding)
\begin{equation}
 \mu_t(A^{c})= \bE\left[ \mu_{\cT_t}(A^c) \right]
 \le \bP\left[ W_t < \frac{(1 \vee \log (1/t))}{\alpha}  \right]+ \bE\left[  \mu_{\cT_t}(A^c)\ind_{\{W_t\ge\frac{(1 \vee \log (1/t))}{\alpha}\}} \right]\,.
\end{equation}
Due to our assumptions $\log (1/t)\le n/10^4$ and $\alpha \ge 1$, one has  $\frac{(1 \vee \log (1/t))}{\alpha}< 1$. By Lemma \ref{important} there exists $c>0$ such that
\begin{equation}
\bP\left[ W_t < \frac{(1 \vee \log (1/t))}{\alpha}\right]\le %{\color{blue} 
2 %\color{black}}
e^{-c\alpha }.
\end{equation}
It remains thus to estimate the second term.
We compute the first two moments of $\Xi_1$ under $\mu_{\cT_t}$.
We have
\begin{equation}
  \mu_{\cT_t}\left(\Xi_1\right) = p(p-1)\sum_{x\in L(\cT_t)} 4^{-|x|}+p= p(p-1)e^{-t/2} W_t+p=(p-1)\frac{r}{\alpha} W_t+p.
\end{equation}
%In the last computation %inequality
% we have ignored rounding effect and we keep doing so in the remainder of the proof for the sake of readability.
Hence on the event $\{W_t\ge\frac{(1 \vee \log (1/t))}{\alpha}\}$ we have 
(note that by definition $p\ge 2$)
%(the equality neglects integer rounding, note that by definition $p\ge 2$)
\begin{equation}\label{40p}
  \mu_{\cT_t}\left(\Xi_1\right)\ge \frac{(p-1) r (1\vee \log(1/t))}{\alpha^2}= 80(p-1)\ge 40 p.
\end{equation}
Next, we compute the second moment. Recalling \eqref{secmondiscrete}, %we have
\begin{multline}
    \mu_{\cT_t}\left(  \Xi^2_1 \right)=
    p(p-1)(p-2)(p-3)     \mu_{\cT_t}\left(  \sigma_1\sigma_2 \sigma_3 \sigma_4\right)\\ + 6p (p-1) (p-2) \mu_{\cT_t}\left(  \sigma^2_1\sigma_2 \sigma_3 \right)+  4p(p-1)  \mu_{\cT_t}\left(  \sigma^3_1\sigma_2\right)\\
    + 3p(p-1) \mu_{\cT_t}\left(  \sigma^2_1\sigma^2_2\right)
    +p  \mu_{\cT_t}\left(  \sigma^4_1\right).
    \end{multline}
    and then
    \begin{multline}
     \mu_{\cT_t}\left(  \Xi^2_1 \right)=
p(p-1)(p-2)(p-3)\left(3\sumtwo{x,y\in L(\cT_t)}{x\ne y} 4^{-|x|}4^{-|y|}+ \sum_{x\in L(\cT_t)} 16^{-|x|}\right)\\ +\left[6p (p-1)(p-2)+ 4p(p-1)\right]  \sum_{x\in L(\cT_t)} 4^{-|x|} +[3(p-1)p+p].
\end{multline}
Therefore,
\begin{equation}
 \left(3\sumtwo{x,y\in L(\cT_t)}{x\ne y} 4^{-|x|}4^{-|y|}+ \sum_{x\in L(\cT_t)} 16^{-|x|}\right)\le3\sum_{x,y\in L_t(\cT)} 4^{-|x|-|y|}= 3  e^{-t}W^2_t,
\end{equation}
and hence,
\begin{equation}\begin{split}
   \mu_{\cT_t}\left(  \Xi^2_1 \right)&\le    3p(p-1)(p-2)(p-3) e^{-t}W^2_t\\
   &\quad + \left[6p (p-1)(p-2)+ 4p(p-1)\right] e^{-t/2}W_t+3p^2\\
   &\le  3p^2 (p-1)^2  e^{-t}W^2_t+6p^2(p-1) e^{-t/2}W_t+3p^2\\
   &= 3\left[(p-1)\frac{r}{\alpha}W_t+p\right]^2.
\end{split} \end{equation}
Thus, by Paley-Zygmund's inequality, on the event  $\{W_t\ge\frac{(1 \vee \log (1/t))}{\alpha}\}$, using  \eqref{40p},
\begin{equation}
 \mu_{\cT_t}\left( \Xi_1 \ge  20p \right)\ge \mu_{\cT_t}\left( \Xi_1 \ge  \frac{1}{2}\mu_{\cT_t}(\Xi_1) \right)
\ge \frac{\mu_{\cT_t}(\Xi_1)^2 }{4  \mu_{\cT_t}\left(  \Xi^2_1 \right)}=\frac{1}{12}.
\end{equation}
The variables $(X_i)_{i=1}^{\alpha}$ thus dominate IID Bernoulli with parameter $1/12$, therefore by a standard large deviation computation there exists a universal constant $c>0$ such that
\begin{equation}
\mu_{\cT_t}\left( A^c \right)\ind_{\{W_t\ge \frac{(1 \vee \log (1/t))}{\alpha}\}}\le e^{-c\alpha}.
\end{equation}
We have shown that $\mu_t( A^c )\le 3e^{-c\alpha}$, which concludes the proof. 
\end{proof}

\begin{acks}[Acknowledgments]
P.C.~and C.L.~thank IMPA for hospitality and financial support during their visit in 2022 where this work was initiated.
\end{acks}

%%%%%%%%%%%%%%%%%%%%%%%%%%%%%%%%%%%%%%%%%%%%%%
%% Funding information, if any,             %%
%% should be provided in the                %%
%% funding section.                         %%
%%%%%%%%%%%%%%%%%%%%%%%%%%%%%%%%%%%%%%%%%%%%%%
\begin{funding}
H.L.\ acknowledges the support of a productivity grant from CNPq and of a CNE grant from FAPERj. C.L.~was partially funded by the ANR project Smooth ANR-22-CE40-0017. \end{funding}

\appendix

\section{Gaussian computations}\label{app:Gauss}

We compute the asymptotics of $\varphi(s) =\|\cN(0,1+s)- \cN(0,1) \|_{\rm TV}$ as $s\downarrow 0$ or $s\uparrow +\infty$. Let $z_s$ be the unique positive real at which the densities of $\cN(0,1)$ and $\cN(0,1+s)$ meet.

In the regime $s\downarrow 0$, $z_s = 1 + o(1)$ and we obtain
\begin{align*}
	\varphi(s) &= \bbP(\cN(0,1+s) \notin [-z_s,z_s]) - \bbP(\cN(0,1) \notin [-z_s,z_s])\\
	&= 2\,\bbP(\cN(0,1+s) > z_s) - 2\bbP(\cN(0,1) > z_s)\\
	&= 2\, \bbP\Big( z_s/\sqrt{1+s} < \cN(0,1) < z_s\Big)\\
	&=2\, \int_{z_s/\sqrt{1+s}}^{z_s} \frac{e^{-\frac{x^2}{2}}}{\sqrt{2\pi}} dx\;.
\end{align*}
A simple computation then shows that
$$ \varphi(s) = \frac{s}{\sqrt{2e\pi}}(1+o(1))\;.$$

In the regime $s\uparrow+\infty$, we have $z_s = \sqrt{\log s} (1+o(1))$ (as $s\uparrow +\infty$). We thus get
\begin{align*}
	\varphi(s) &= \bbP(-z_s \le \cN(0,1) \le z_s) - \bbP(-z_s \le \cN(0,1+s) \le z_s)\\
	&= 1-\bbP(\vert\cN(0,1)\vert > z_s) - \bbP(\vert \cN(0,1)\vert \le \frac{z_s}{\sqrt{1+s}})\;.
\end{align*}
We then compute
$$ \bbP(\vert\cN(0,1)\vert > z_s) \le \frac2{\sqrt{2\pi s \log s}}\;,$$
while
$$ \bbP(\vert \cN(0,1)\vert \le \frac{z_s}{\sqrt{1+s}}) = \frac{2}{\sqrt{2\pi}} \frac{z_s}{\sqrt{1+s}} (1+o(1)) = \frac{2\sqrt{\log s }}{\sqrt{2\pi s}} (1+o(1))\;,$$
so that
$$ 1 - \varphi(s) =  \frac{2\sqrt{\log s }}{\sqrt{2\pi s}} (1+o(1))\;,\quad s\uparrow +\infty\;.$$

\section{A tricky inequality}\label{app:tricky}
The next lemma  proves the inequality \eqref{eq:finerineq}.
\begin{lemma}
	Let $\pi$ be a probability measure on some measurable space $(\O,\cA)$. Let $f:\gO\to \bbR_+$ be a density w.r.t.~$\pi$. Then
	\begin{equation}\label{eq:fineq}
		\tfrac12 \,  \|f - 1\|_{L^1(\pi)} \le \phi \left(\|f - 1\|_{L^2(\pi)} \right)
	\end{equation}
	where 
	$$ \phi(x)=\begin{cases} 
		\frac{x}{2},  &\text{ if } x\le 1,\\
		\frac{x^2}{1+x^2} \quad &\text{ if } x\ge1.
	\end{cases}$$
\end{lemma}
The reader can check that the proof below implies that the inequality is sharp in the sense that if $\pi$ has no atom,
then for any value $x > 0$, it is possible to find some density function $f$ such that $\|f - 1\|_{L^2(\pi)}=x$ and \eqref{eq:fineq} is an equality.
\begin{proof}
We recall that $\|f - 1\|_\tv= \tfrac12\|f - 1\|_{L^1(\pi)}$. We are going to prove the (equivalent) inverse inequality, that is
\begin{equation}\label{desired}
\|f - 1\|_{L^2(\pi)}\ge \phi^{-1}\left(\|f - 1\|_\tv\right)\;,
\end{equation}
where
	$$ \phi^{-1}(u)=\begin{cases} 
	2u,  &\text{ if } u\le 1/2,\\
	\sqrt{\frac{u}{1-u}} \quad &\text{ if } u\ge1/2.
\end{cases}$$
Assume that the inequality holds for functions that assume only two values, i.e.\ functions of the form
\begin{equation}\label{simpleform}
	f=(1+a)\ind_A+ \left(1-\frac{a\pi(A)}{1-\pi(A)}\right)\ind_{A^{c}}
\end{equation}
where necessarily $a\in [0,\frac{1-\pi(A)}{\pi(A)}]$ and $A\in \cA$. Then for a generic density function $f$, we introduce the density function $\bar f$ by setting
\begin{equation}
	 \bar f = \frac{\int f\ind_{\{f\ge 1\}} \dd \pi}{\int \ind_{\{f\ge 1\}} \dd \pi} \ind_{\{f\ge 1\}}+ \frac{\int f\ind_{\{f< 1\}} \dd \pi}{\int \ind_{\{f\ge 1\}} \dd \pi} \ind_{\{f< 1\}}.
\end{equation}
It is straightforward to check that
$$ \|f - 1\|_\tv= \|\bar f - 1\|_\tv   \quad  \text{ and } \quad  \|f - 1\|_{L^2(\pi)}\ge \|\bar f - 1\|_{L^2(\pi)}\;.$$
Since the density function $\bar f$ only takes two values, it satisfies the inequality of the statement and we deduce that
$$ \|f - 1\|_{L^2(\pi)}\ge \|\bar f - 1\|_{L^2(\pi)} \ge \phi^{-1}\left(\|f - 1\|_\tv\right) = \phi^{-1}\left(\|f - 1\|_\tv\right)\;,$$
as required. We are left with checking the inequality in the case where $f$ is of the form \eqref{simpleform}.\\
In that case we have
\begin{equation}
  \|f - 1\|_\tv=a\pi(A)
\end{equation}
and
\begin{equation}
\|f - 1\|_{L^2(\pi)}= \sqrt{\frac{a^2\pi(A)}{1-\pi(A)}}= \frac{\|f - 1\|_\tv}{\sqrt{\pi(A)(1-\pi(A))}}.
\end{equation}
The above mentioned constraint on the parameter $a$ implies that $\pi(A)\le 1- \|f - 1\|_\tv$. Consequently
\begin{equation}
 \|f - 1\|_{L^2(\pi)}\ge \min_{u\in [0, 1- \|f - 1\|_\tv]}\frac{\|f - 1\|_\tv}{\sqrt{u(1-u)}}.
\end{equation}
The minimum is attained at $u=1/2$ if $\|f - 1\|_\tv\le 1/2$ and at $u=1- \|f - 1\|_\tv$ otherwise. Hence we obtain the desired inequality \eqref{desired}.
\end{proof}

\section{Poissonian computations}\label{app:Poisson}

Let $X$ and $Y$ be two Poisson processes of intensity $1$ and $1/2$ respectively. Fix $t>0$. We aim at proving that for every measurable and bounded map $G$ defined on the Skorohod's space of c\`adl\`ag processes on $[0,t]$, the following identity holds
$$ \bbE\left[e^{t/2}2^{-X_t}G(X)\right]=\bbE\left[G(Y)\right] \;.$$
To that end, it suffices to prove that for any integer $k\ge 1$, for all $0=t_0 \le t_1 < \ldots < t_k = t$ and all $q_1,\ldots,q_k \in \R$
$$ \bbE\left[ e^{t/2}2^{-X_t} \prod_{j=1}^k e^{-q_j X_{t_j}}\right] = \prod_{j=1}^k \exp\left(\frac12(t_j-t_{j-1})(e^{-q_j-\ldots-q_k} - 1)\right)\;.$$

Set $\tilde{q}_j := q_j$ for all $j< k$ and $\tilde{q}_k := q_k + \log 2$. Using the independence and stationarity of the increments of a Poisson process, we find
\begin{align*}
	\bbE\left[ e^{t/2}2^{-X_t} \prod_{j=1}^k e^{-q_j X_{t_j}}\right] &= e^{t/2} \bbE\left[ \prod_{j=1}^k e^{-(\tilde{q}_j+\ldots+\tilde{q}_k) (X_{t_j} - X_{t_{j-1}})}\right]\\
	&=e^{t/2} \prod_{j=1}^k \exp\left((t_j-t_{j-1})(e^{-\tilde{q}_j-\ldots-\tilde{q}_k} - 1)\right)\\
	&=e^{t/2} \prod_{j=1}^k \exp\left((t_j-t_{j-1})(\frac12 e^{-{q}_j-\ldots-{q}_k} - 1)\right)
\end{align*}
Since $e^{t/2} = \prod_{j=1}^k \exp(\frac12 (t_j-t_{j-1}))$, we easily conclude. \qed

\section{Proof of Lemma \ref{important}}\label{sec:important}
We start with the upper bound. The martingale property together with Markov's inequality show that $\bP(W_{\infty}\le 2\gep \ | \ W_t\le \gep )\ge 1/2$ and therefore \[\bP\left( W_t \le \gep\right) \le 2 \bP\left( W_\infty \le 2\gep\right).\] 
We will use the above estimate when $t>1$, in which case we can further assume that  $\gep$ is sufficiently small.
% 
%Hence we need to prove the bound only for $t=\infty$ and $t\in(0,1]$.
For any  $k\in\bbN$, a.s.\
\begin{equation}
W_\infty\ge 4^{-k} \sum_{i=1}^{2^k} W^{(i)}_\infty,
\end{equation}
where $(W^{(i)}_\infty)_{i=1}^{2^k}$ are the limits of the martingales corresponding to the trees rooted at the $2^k$ leaves of generation $k$ in the tree $\bbT$.  Hence
\begin{equation}
\bP\left( W_\infty \le 2^{-(k+1)} \right) \le \bP\left( \sum_{i=1}^{2^k} W^{(i)}_\infty \le 2^{k-1}\right).
\end{equation}
The variables  $(W^{(i)}_\infty)_{i=1}^{2^k}$ are IID with the same distribution as $W_\infty$. Since $\bE[W_\infty]=1$, the above is a large deviation event for a sequence of IID random variables and thus has a probability smaller that $\exp(-c 2^{k})$ for some constant $c>0$. Hence, taking e.g.\ $k$ such that $2^k\gep\in[1/8,1/4]$, the previous estimates are sufficient to prove the desired upper bound in Lemma \ref{important} for all $t>1$, and all $\gep\in(0,1)$. 

When $t\in(0,1]$, then %we first observe that 
$W_t<1$ if and only if there has been a splitting at the root at time $t$. Therefore, % we have
\begin{equation}\label{largeeps}
 \bP[W_t<1]=\bP[\tau_{\emptyset}<t]=1-e^{-t}.
\end{equation}
For the remainder of the proof we can assume that $\gep\le \gep_0$ sufficiently small since \eqref{largeeps} allows us to deal with $\gep \in [\gep_0,1)$, by tuning the constant in the inequality appropriately.
We notice that since $W_t= e^{t/2}E_{\cT_t}[2^{-|U_1|}]$, by Markov's inequality $W_t\le \gep$ only if
$$P_{\cT_t}\left[ e^{t/2} 2^{-|U_1|}\le 4\gep \right]\ge \frac{3}{4}.$$
The above implies that a portion at least $3/4$ of the vertices at generation ${\bf h}_{\gep,t}:= \lceil -\log_2(4\gep)+\frac{t}{2\log 2}\rceil$ have to be in $\cT_t$. This implies in particular (recall \eqref{defct}) that the exponential clocks corresponding to the parents of those vertices are smaller than $t$, or in other words that
\begin{equation}
 \#\{ x\in \bbT \ : \ |x|={\bf h}_{\gep,t}-1 \text{ and } \cE_x\le t \} \ge 3 \times 2^{{\bf h}_{\gep,t}-3}.
\end{equation}
To conclude we only need to prove that 
\begin{equation}\label{zelast}
\bP\left[ \#\{ x\in \bbT \ : \ |x|={\bf h}_{\gep,t}-1 \text{ and } \cE_x\le t \} \ge 3 \times 2^{{\bf h}_{\gep,t}-3} \right] \le e^{-\frac{c(1\vee \log (1/t))}{\gep}}.
\end{equation}
First observe that by Cram\'er's Theorem (the cardinality to estimate is a sum of Bernoulli variables of parameter $1-e^{-t}$, and note that $3/4> 1-e^{-1}$)
we have 
\begin{equation}
 \bP\left[ \#\{ x\in \bbT \ : \ |x|={\bf h}_{\gep,t}-1 \text{ and } \cE_x\le 1 \} \ge 3 \times 2^{{\bf h}_{\gep,t}-3} \right] \le e^{- c 2^{{\bf h}_{\gep,t}}}\le  e^{-\frac{c'}{\gep}},
\end{equation}
for suitable constants $c,c'>0$. Thus, to conclude we only need to prove that \eqref{zelast} holds for sufficiently small values of $t$.
In that case, we observe that $\#\{ x\in \bbT \ : \ |x|={\bf h}_{\gep,t}-1 \text{ and } \cE_x\le t \}$ is a Binomial r.v.~with parameters $4N$ and $1-e^{-t}$, where $N= 2^{{\bf h}_{\gep,t}-3}$, and we thus use the following rough bound on the binomial distribution: 
\begin{equation}
 \bP\left[ \mathrm{Bin}(1-e^{-t},4N)\ge 3N \right]\le %{\color{blue} N}\color{black}
 \binom{4N}{3N}(1-e^{-t})^{3N}
 \le e^{3N\left(\log t\right)+CN},
\end{equation}
for a suitable constant $C$.
This ends the proof of the upper bound in Lemma \ref{important}.

% \cyril{Je reprends le paragraphe ci-dessous car j'ai qq doutes:}
% 
% 
% % and hence that a proportion $3/4$ of the vertices $y$ at generation  ${\bf h}_{\gep,t}-1$ satisfy $\cE_y\le t$.
% % 
% % 
% % The number of such clocks is $2^{{\bf h}_{\gep,t}}$.
% 
% Recall that $\tau_x$ follows a $\Gamma(|x|,1)$ law. Let $X_i$, $i=1,\ldots,2^{{\bf h}_{\gep,t}}$ be IID Bernoulli r.v.~with parameters $p_{\gep,t}=\bbP(\Gamma({\bf h}_{\gep,t},1) \le t)$. Note that
% $$ p_{\gep,t} = \bbP(\Gamma({\bf h}_{\gep,t},1) \le t) \le c \frac{t^{{\bf h}_{\gep,t}}}{{\bf h}_{\gep,t}!}\;.$$
% We have
% \begin{equation}
% 	\bP\left[ P_{\cT_t}\left[ e^{t/2} 2^{-|U_i|}\le 4\gep \right]\ge \frac{3}{4}\right] \le  \bP\left[ \sum_{i=1}^{2^{{\bf h}_{\gep,t}}} X_i \ge \frac{3}{4} 2^{{\bf h}_{\gep,t}}  \right].
% \end{equation}
% Large deviation estimates for sums of IID Bernoulli show that this last term is bounded by a term of order
% $$ \exp(-c'2^{{\bf h}_{\gep,t}} \ln(\frac1{p_{\gep,t}})) \le \exp(-c'' \frac{1 \vee \log(1/t)}{\gep})\;,$$
% uniformly over all $t\in (0,1]$. 
% \cyril{fin du paragraphe}

For the lower-bound, we start with the case $t\in(0,1]$. We consider the event $A^{(t)}_k$ that for generations numbered from $0$ to $k-1$, the splitting times are smaller than $t2^{|x|-k}$, that is to say, the root splits after a time smaller than $t2^{-k}$, its descendent split after an additional time smaller than $(t2^{1-k}$ etc...).
Using $1-e^{-t}\ge t/2$ for $t\in[0,1]$,
\begin{align*}
 \bP(A^{(t)}_k)=\prod_{i=0}^{k-1} (1-e^{-t2^{i-k}})^{2^{i}}&\ge t^{2^k-1} 2^{-\sum_{i=0}^{k-1}(k+1-i)2^i  }\\
 &\ge t^{2^k-1} 2^{-2^{k+1} \sum_{j\ge 0}j 2^{-j} } \le t^{2^k-1} 2^{-c 2^{k+1}}\;,
\end{align*}
for some constant $c>0$.
On the event $A^{(t)}_k$,  $\cT_{t}$ has a complete $k$-th generation. Therefore, 
\begin{equation}
 W_{t}\le e^{1/2} 2^{-k}.
\end{equation}
Choosing $k=k_{\gep}$ such that $e^{1/2} 2^{-k}\in (\gep/2,\gep]$ we obtain that for $t\le 1$, 
\begin{equation}\label{boundonak}
\bP[W_t\le \gep]\ge \bP(A_{k})\ge \exp\left( - c'(\log (1/t)+1)2^{k}\right) 
\end{equation}
and we can conclude using the fact that $2^{k}$ is of order $\gep^{-1}$. This settles the lower bound for $t\in(0,1]$.

When $t>1$, using the martingale property at time $1$ and Markov's inequality,
\begin{equation}
 \bP\left[ W_t \le \gep   \ | \ A^{(1)}_k\right]\ge 1- \frac{e^{1/2} 2^{-k}}{\gep}.
\end{equation}
Hence for every $k\ge 0$
\begin{equation}
  \bP\left[ W_t \le \gep \right] \ge P(A^{(1)}_k)\left(1- \frac{e^{1/2} 2^{-k}}{\gep}\right).
\end{equation}
and we conclude by taking e.g.\ $k=k'_{\gep}\ge 1$ such that   $e^{1/2} 2^{-{k'_{\gep}}}\in (\frac{9\gep}{20},\frac{9\gep}{10}]$.
This ends the proof of the lower bound in Lemma \ref{important}.
\qed

\bibliographystyle{imsart-number} % Style BST file (imsart-number.bst or imsart-nameyear.bst)
\bibliography{nonlinear}

\end{document}